\documentclass[reqno,11pt]{amsart}

\usepackage[a4paper,left=24mm,right=24mm,top=30mm,bottom=30mm,marginpar=20mm]{geometry} 
\usepackage{amsmath}
\usepackage{amssymb}
\usepackage{amsthm}
\usepackage{amscd}
\usepackage[ansinew]{inputenc}
\usepackage{color}
\usepackage[final]{graphicx}
\usepackage{tikz}
\usepackage{bbm}

\usetikzlibrary{patterns} 
\usetikzlibrary{positioning}
\usetikzlibrary{calc}
\usetikzlibrary{arrows,shapes,backgrounds}
\usetikzlibrary{intersections}

\renewcommand{\epsilon}{\varepsilon}

\numberwithin{equation}{section}

\newtheoremstyle{thmlemcorr}{10pt}{10pt}{\itshape}{}{\bfseries}{.}{10pt}{{\thmname{#1}\thmnumber{ #2}\thmnote{ (#3)}}}
\newtheoremstyle{thmlemcorr*}{10pt}{10pt}{\itshape}{}{\bfseries}{.}\newline{{\thmname{#1}\thmnumber{ #2}\thmnote{ (#3)}}}
\newtheoremstyle{defi}{10pt}{10pt}{\itshape}{}{\bfseries}{.}{10pt}{{\thmname{#1}\thmnumber{ #2}\thmnote{ (#3)}}}
\newtheoremstyle{remexample}{10pt}{10pt}{}{}{\bfseries}{.}{10pt}{{\thmname{#1}\thmnumber{ #2}\thmnote{ (#3)}}}
\newtheoremstyle{ass}{10pt}{10pt}{}{}{\bfseries}{.}{10pt}{{\thmname{#1}\thmnumber{ A#2}\thmnote{ (#3)}}}

\theoremstyle{thmlemcorr}
\newtheorem{theorem}{Theorem}
\numberwithin{theorem}{section}
\newtheorem{lemma}[theorem]{Lemma}

\theoremstyle{thmlemcorr*}
\newtheorem{theorem*}{Theorem}
\newtheorem{lemma*}[theorem]{Lemma}
\newtheorem{corollary*}[theorem]{Corollary}
\newtheorem{proposition*}[theorem]{Proposition}
\newtheorem{problem*}[theorem]{Problem}
\newtheorem{conjecture*}[theorem]{Conjecture}

\theoremstyle{defi}

\theoremstyle{remexample}
\newtheorem{remark}[theorem]{Remark}

\theoremstyle{ass}

\newcommand{\Ecal}{\mathcal{E}}

\newcommand{\Ical}{\mathcal{I}}

\newcommand{\Mcal}{\mathcal{M}}

\newcommand{\Ocal}{\mathcal{O}}

\DeclareMathOperator{\dist}{dist}

\newcommand{\norm}[1]{\|#1\|}

\newcommand{\dd}{\;\mathrm{d}}

\newcommand{\N}{\mathbb{N}}
\newcommand{\R}{\mathbb{R}}

\newcommand{\Z}{\mathbb{Z}}

\newcommand{\weakly}{\rightharpoonup}

\newcommand{\eps}{\epsilon}

\newcommand{\ffi}{\varphi}

\DeclareMathOperator{\SO}{SO}
\DeclareMathOperator{\Sl}{Sl}

\def\Xint#1{\mathchoice 
{\XXint\displaystyle\textstyle{#1}}%
{\XXint\textstyle\scriptstyle{#1}}%
{\XXint\scriptstyle\scriptscriptstyle{#1}}%
{\XXint\scriptscriptstyle\scriptscriptstyle{#1}}%
\!\int} 
\def\XXint#1#2#3{{\setbox0=\hbox{$#1{#2#3}{\int}$} 
\vcenter{\hbox{$#2#3$}}\kern-.5\wd0}} 
\def\dashint{\,\Xint-}

\AtBeginDocument{}
\newcommand{\rn}[1]{\nabla^\epsilon #1_\epsilon}
\newcommand{\ui}[1]{^{\left(#1\right)}}
\newcommand{\Wmin}{\overline{W}}
\newcommand{\pw}{\mathrm{pw}}
\newcommand{\Affpw}{A_{\pw}(0,L;\R^3)}

\newcommand\restrict[1]{\raisebox{-.5ex}{$|$}_{#1}}

\title[Incompressible elastic strings]{Asymptotic variational analysis of \\ incompressible elastic strings}

\author{Dominik Engl}
\address{Mathematisch Instituut, Universiteit Utrecht, Postbus 80010, 3508 TA Utrecht, The Netherlands}
\email{D.M.Engl@uu.nl}

\author{Carolin Kreisbeck}
\address{Mathematisch Instituut, Universiteit Utrecht, Postbus 80010, 3508 TA Utrecht, The Netherlands}
\email{C.Kreisbeck@uu.nl}

\begin{document}

\maketitle
\thispagestyle{empty}

\maketitle

 \begin{abstract}  
 \vspace{-12pt}  
Starting from three-dimensional nonlinear elasticity under the restriction of incompressibility, we derive reduced models to capture the behavior of strings in response to external forces.
Our $\Gamma$-convergence analysis of the constrained energy functionals in the limit of shrinking cross sections gives rise to explicit one-dimensional limit energies. 
The latter depend on the scaling of the applied forces. The effect of local volume preservation is reflected either in their energy densities through a constrained minimization over the cross-section variables or in the class of admissible deformations. 
Interestingly, all scaling regimes allow for compression and/or stretching of the string.  
The main difficulty in the proof of the $\Gamma$-limit is to establish recovery sequences that accommodate the nonlinear differential constraint imposed by the incompressibility. 
To this end, we modify classical constructions in the unconstrained case with the help of an inner perturbation argument tailored for $3$d-$1$d dimension reduction problems.                        
\vspace{8pt}

 

 \noindent\textsc{Date:} \today.
 \end{abstract}

\section{Introduction}

Modern mathematical approaches to applications in materials science result in variational problems with non-standard constraints for which the classical methods of the calculus of variations do not apply. 
Constraints involving non-convexity, differential expressions and/or nonlocal effects are known to be particularly challenging.

In the context of the analysis of thin objects, interesting effects may occur due to the interaction between restrictive material properties and the lower-dimensional structure of the objects. We mention here a few selected examples:  
thin (heterogenous) films and strings subject to linear first-order partial differential equations, which are general enough to cover applications in nonlinear elasticity and micromagnetism at the same time, are studied in~\cite{Kre17, KrK16, KrR15}, cf.~also~ \cite{GiJ97, Kre13}; pointwise constraints on the stress fields appear naturally in models of perfectly plastic plates \cite{Dav14, DaM13}; for work on lower-dimensional material models that involve issues related to non-interpenetration of matter and (global) invertibility, we refer for instance to \cite{LPS15, OlR17, Sch07, Zor06}; physical growth conditions, which guarantee orientation preservation of deformation maps, 
have been taken into account in models of thin nematic elastomers~\cite{AgD17a} and von K\'arm\'an type rods and plates~\cite{DaM12, MoS12}.

This paper is concerned with $3$d-$1$d dimension reduction problems in nonlinear elasticity with incompressibility  -  a determinant constraint on the deformation gradient, which ensures local volume preservation, and is ideal to model e.g.~rubber-like materials~\cite{Ogd72}. 
To be more specific, we provide an ansatz-free derivation of reduced models for incompressible thin tubes by means of $\Gamma$-convergence techniques (see~\cite{Bra02, Dal93} for a comprehensive introduction). We take the limit of vanishing cross section, considering external loading of the order of magnitude that gives rise to string type models. 

The analogous problem in the $3$d-$2$d context, meaning for incompressible membranes, was solved independently by Trabelsi~\cite{Tra06} and Conti \& Dolzmann~\cite{CoD06} based on different approaches. To overcome the difficulty of accommodating the nonlinear differential constraint when constructing recovery sequences,~\cite{CoD06} involves the construction of suitable inner variations. This idea has been applied in the analysis of incompressible Kirchhoff and von K\'arm\'an plates~\cite{CoD09, ChL13}, and lends inspiration to this paper, where we adapt it for $3$d-$1$d reductions.

The first results in the literature to use $\Gamma$-convergence techniques to deduce reduced models for thin objects go back to the 1990s, with the seminal works by Acerbi, Buttazzo \& Percivale \cite{ABP91} on strings and Le Dret \& Raoult \cite{LeR95} on membranes. Notice that in both papers, the authors start from unconstrained energy functionals whose energy densities satisfy standard growth. 
Before that, common techniques for gaining quantitative insight into thin structures relied mostly on asymptotic expansion methods, and were applied in the setting of linearized elasticity, see e.g.~\cite{Cia97, TrV96}. 

Over the last two decades, the fundamental results in~\cite{ABP91, LeR95} have been generalized in multiple directions. This includes for instance the study of membrane theory with Cosserat vectors~\cite{BFM03, BFM09}, curved strings~\cite{Sca06}, inhomogeneous thin films~\cite{BFF00}, thin structures made of periodically heterogeneous material \cite{BaB06, BaF05, LeM05}, or junctions between membranes and strings~\cite{FeZ19}. 

\subsection{Problem formulation}

For small $\eps>0$, let $\Omega_\eps:=(0,L)\times \eps\omega$ with $L>0$ and a bounded Lipschitz domain $\omega\subset \R^2$ represent the reference configuration of a thin unilaterally extended body. Up to translation, we may always assume that the origin lies in $\omega$. 

The starting point of our analysis is a three-dimensional model in hyperelasticity with an energy functional (per unit cross section) of the form
\begin{align*}
	\Ecal_\epsilon(v) = \frac{1}{\eps^2}\int_{\Omega_\epsilon} W(\nabla v) \dd y - \frac{1}{\eps^2}\int_{\Omega_\epsilon} f_\epsilon\cdot v \dd y,\quad v\in H^1(\Omega_\epsilon;\R^3);
\end{align*}
here, $f_\eps \in L^2(\Omega_\eps;\R^3)$ are external forces and $W$ is a constrained stored elastic energy density enforcing  incompressibility, precisely,
\begin{align}\label{WW0}
	W: \R^{3\times 3}\to [0,\infty],\quad F\mapsto\begin{cases} W_0(F) &\text{ if } \det F = 1,\\ \infty &\text{ otherwise,}\end{cases}
\end{align} 
with $W_0 :\R^{3\times 3}\to [0,\infty)$ a continuous function with suitable growth behavior. We give more details on the exact assumptions on $W_0$ in Section~\ref{subsec:hypotheses}, see (H1)-(H3).
In this model, the observed deformations of the thin object in response to external forces correspond to minimizers (or, if the latter do not exist, low energy states) of $\Ecal_\eps$.

To derive reduced one-dimensional models that capture the asymptotic behavior of these minimizers, it is technically convenient to work with functionals defined on $\eps$-independent spaces, which can be achieved by a classical rescaling argument in the cross section.
Indeed, let $u_\eps(x):=v(y)$ for $v\in H^1(\Omega_\eps;\R^3)$ with the parameter transformation $y=(x_1, \eps x_2, \eps x_3)$ for $x\in \Omega:=\Omega_1$, and 
suppose for simplicity that $f_\eps$ is independent of the cross-section variables. Then,
 $\Ical_\eps(u_\eps) = \Ecal_\eps(v)$, where
\begin{align*}
	 \Ical_\eps(u) := \int_{\Omega} W\bigl(\rn{u}\bigr) \dd x - \int_{\Omega} f_\eps \cdot u \dd x, \qquad u\in H^1(\Omega;\R^3),
\end{align*}
with $\rn{u} = (\partial_1 u|\tfrac{1}{\epsilon}\partial_2 u|\tfrac{1}{\epsilon} \partial_3 u) $ the rescaled gradient of $u$. 

In analogy to the well-known facts in the context of compressible materials (see e.g.~\cite{FJM06}), here as well, the scaling behavior of $\Ical_\eps$ depends strongly on the external forces $f_\eps$. 
Whenever 
 $f_\epsilon$ is of order $\epsilon^\alpha$ for some $\alpha\geq 0$, then $\inf_{u\in H^1(\Omega;\R^3)}\Ical_\eps(u)$ behaves like $\epsilon^\beta$ with $\beta = \alpha$ if $\alpha\leq 2$ and $\beta=2\alpha-2$ if $\alpha \geq 2$. Depending on these scalings, one has to expect qualitatively different limit models, falling into the categories of string theory ($\alpha=0$), rod theories ($\alpha=2$ and $\alpha =3$) or other intermediate theories.

Since this article deals with the regimes $\alpha<2$ (the cases $\alpha\geq 2$ are addressed in a different work, see~\cite{EnK19}), it is natural to consider in the following the rescaled functionals $\Ical_\eps^\alpha: H^1(\Omega;\R^3)\to [0,\infty]$ for $\alpha\in [0,2)$ defined by
\begin{align}\label{energy_all_scaling_regimes}
		\Ical_\eps^\alpha(u) = \frac{1}{\epsilon^\alpha}\int_\Omega W(\nabla^\epsilon u) \dd x, \quad u\in H^1(\Omega;\R^3);
\end{align}
notice that one may, without loss of generality, omit here the term describing work due to the external forces, for it is merely a continuous perturbation of the total (rescaled) elastic energy.

\subsection{Statement of the main results.}
The new contribution of this paper is a complete characterization of the $\Gamma$-limits of sequences $(\Ical^\alpha_\epsilon)_{\epsilon}$ as in \eqref{energy_all_scaling_regimes} for $\eps\to 0$.  

To be more precise, we prove that under suitable assumptions,
 $(\Ical^\alpha_\eps)_{\eps}$ $\Gamma$-converges with respect to the weak topology in $H^1(\Omega;\R^3)$ to $\Ical^\alpha: H^1(\Omega;\R^3)\to [0, \infty]$ given for $\alpha=0$ by
\begin{align}\label{Ical0}
	\Ical^0(u)= \begin{cases}  |\omega| \displaystyle\int_0^L \Wmin^{\rm c}(u'(x_1)) \dd x_1 &\text{ if } u \in H^1(0,L;\R^3), \\ \infty &\text{ otherwise.}\end{cases}
\end{align}
and for $\alpha\in (0,2)$ by
\begin{align}\label{Ical_alpha}
	\Ical^\alpha(u) = \begin{cases}  0 &\text{ if  $u \in H^1(0,L;\R^3)$ with $\Wmin^{\rm c}(u'(x_1))=0$ for a.e.~$x_1\in (0, L)$,}  \\ \infty &\text{ otherwise,}\end{cases}
\end{align}
respectively, cf.~Theorem~\ref{theo:strings=0} and~\ref{theo:strings>0} for all details.

The reduced energy density $\Wmin$ results from minimizing out the cross-section variables from $W$, that is,
\begin{align*}
\Wmin(\xi) := \min_{A\in \R^{3\times 2}} \Wmin \bigl((\xi|A)\bigr) = \min_{A\in\R^{3\times 2}, \,\det (\xi| A)=1} W_0 \bigl((\xi|A)\bigr), \qquad \xi\in \R^3\setminus \{0\},
\end{align*}
cf.~\eqref{overlineW}, while the convexification $\Wmin^{\rm c}$ of $\Wmin$ reflects a relaxation process. 

The representation formulas~\eqref{Ical0} and~\eqref{Ical_alpha} indicate that the two regimes $\alpha=0$ and $\alpha\in (0,2)$ give rise to qualitatively different reduced one-dimensional models. Whereas the latter admits only restricted deformations of the thin object, which can however be obtained with zero energy, the former allows us for any deformation of the string at finite energetic cost.

Despite their differences, both cases share a feature that may seem surprising at first. In fact, the incompressibilty constraint imposed on the three-dimensional elasticity models does not carry over to the reduced ones, in the sense that admissible deformations are not length preserving in general, but can undergo compression and/or stretching. 
For a similar observation in the context of incompressible membranes, see~\cite{CoD06}.

\subsection{Approach and techniques}  	The proofs for the cases $\alpha=0$ and $\alpha\in (0,2)$ can be found in Section~\ref{sec:alpha=0} and Section~\ref{sec:alpha>0}, respectively. 
Overall, our idea is to combine tools from~\cite{CoD06} on $3$d-$2$d dimension reduction for incompressible membranes and the 
references \cite{ABP91, Sca06}, where the authors derive one-dimensional models for strings without  volumentric constraints.

In both regimes, compactness and the liminf-inequalities are straightforward to show, as they follow immediately from the corresponding results for the unconstrained problems.
However, the construction of recovery sequences is more delicate.  

The difficulty is to accommodate the incompressibility condition, while approximating the desired limit deformation in an energetically optimal way.  To achieve this, we take the recovery sequences from the compressible case - i.e., the ones from~\cite{ABP91} if $\alpha=0$, and from~\cite{Sca06} for $\alpha\in (0, 2)$ - as a basis, and modify them with the help of an inner perturbation argument tailored for $3$d-$1$d dimension reduction. The latter, which is stated in Lemma~\ref{lem:reparam_det=1}, is a key ingredient of the proof. 

In order to apply Lemma~\ref{lem:reparam_det=1},  though, one needs sequences that are sufficiently regular and whose rescaled deformation gradients have determinant close to $1$ up to a small, quantified error. Especially in the string regime $\alpha=0$, this requires some technical effort. Indeed, with the help of B\'ezier curves, we establish a mollification argument for piecewise affine functions of one variable, which, amongst other useful properties, yields uniform bounds on the derivatives, see Lemma~\ref{prop:bezier_mollifying}. 
Moreover, we construct tailored moving frames along the resulting smooth curves in order to guarantee that fattening them to tubes results in deformed configurations that are almost (with controlled errors) locally volume-preserving.

\newpage
\section{Preliminaries}
To start with, we introduce notations and collect a few technical tools.
\subsection{Notation}
	The following notational conventions are used throughout the paper:  Let $a\cdot b$ be the standard inner product of two vectors $a,b\in\R^3$, and 
	$e_1, e_2, e_3$ the standard unit vectors in $\R^3$. 
	On the space $\R^{m\times n}$ of real-valued $m\times n$ matrices, we denote the Forbenius norm by $|\cdot|$. 
	
	Moreover, the closure of a given set $U\subset \R^n$ is denoted by $\overline U$, whereas $U^{\rm c}$ stands for the convex hull of $U$. Accordingly, the convex envelope of a function $f: \R^n\to \R$ is $f^{\rm c}$.  For the zero level set of $f$, we use $L_0(f)$.
	
	The partial derivative of $v:U \to \R^m$, where $U\subset \R^n$ is open, with respect to the $i$-th variable is denoted by $\partial_i v$, and gradients by $\nabla v$. If $v$ depends only on one real variable, meaning if $n=1$, 
	we simplify $\partial_1 v$ to $v'$.  The rescaled gradient of $v:U\subset \R^3 \to \R^3$ is given by
	\begin{align}\label{rescaled}
			\nabla^\epsilon v := \bigl(\partial_1 v \big| \tfrac{1}{\epsilon}\partial_2 v\big| \tfrac{1}{\epsilon}\partial_3  v\bigr).
	\end{align}

	We employ standard notation for Lebesgue and Sobolev spaces as well as for spaces of continuous and $k$-times continuously differentiable functions; in particular, $L^2(U;\R^m)$, $H^1(U;\R^m)$ and $C^k(\overline U;\R^m)$ with $k\in \N_0$ for an open set $U\subset \R^n$. We shall equip the latter with the norm 
	\begin{align*}
	\textstyle \norm{u}_{C^k(\overline U;\R^m)} := \sum_{i=0}^k\max_{x\in \overline U} |\nabla^i u(x)|.
	\end{align*}	
To shorten notation, let $L^2(a,b;\R^m):=L^2((a,b);\R^m)$ and $H^1(a,b;\R^m):=H^1((a,b);\R^m)$ if $U=(a, b)\subset \R$.  

Furthermore, $A_{\rm pw}(I;\R^m)$ with an open interval $I\subset \R$ is defined as the space of continuous piecewise affine functions with values in $\R^m$.

If $\Omega=(0, L)\times \omega$ as in the introduction, without explicit mention, we identify any $u:(0,L)\to \R^3$ with its trivial extension to a function on $\Omega$; in particular, given sufficient regularity, $\partial_1 u$ and $u'$ are used interchangeably.
	
Finally, $\Ocal(\cdot)$ is the well-known Landau symbol. 

	\subsection{Hypotheses and properties of $W_0$ and $\Wmin$}\label{subsec:hypotheses}

Consider the following regularity and growth assumptions for the energy density $W_0 :\R^{3\times 3} \to [0,\infty)$: 
	\begin{itemize}
		\item[(H1)] $W_0$ is continuous on $\R^{3\times 3}$; 
		\item[(H2)] there are constants $C_2, c_2>0$ such that 
			\begin{align*}
				c_2|F|^2 - C_2 \leq W_0(F) \leq C_2(|F|^2 + 1)\quad \text{for all $F\in \R^{3\times 3}$};
			\end{align*}
		\item[(H3)] there are constants $C_3, c_3>0$ such that 
			\begin{align*}
				c_3\,{\rm dist}^2(F,\SO(3)) \leq W_0(F) \leq C_3\,{\rm dist}^2(F, \SO(3))\quad \text{for all $F\in \R^{3\times 3}$}. 
			\end{align*}
	\end{itemize}	
	Clearly, (H3) implies (H2). 
	
 Recalling the definition of $W$ in~\eqref{WW0}, we define $\overline{W}:\R^{3}\to [0, \infty]$ by minimizing out the cross-section variables, that is, 
	\begin{align}\label{overlineW}
			\Wmin(\xi) = \inf_{A \in \R^{3\times 2}} W\bigl((\xi|A)\bigr) = \begin{cases} \inf_{A \in \R^{3\times 2},\,\det(\xi|A) = 1} W_0\bigl((\xi|A)\bigr) & \text{if $\xi\neq 0$,}\\ \infty & \text{otherwise, } \end{cases}
	\end{align} 
		for $\xi\in \R^3$. Notice that the hypotheses (H1) and (H2) guarantee that the infima in~\eqref{overlineW} are attained.
	
	 As we detail next, the convexification $\overline{W}^{\rm c}$ of $\Wmin$ inherits growth properties from $W_0$.
	\begin{lemma}
		Let $W_0$ satisfy (H1) and (H2). There are constants $C_4,c_4>0$ such that 
		\begin{align}\label{convex_energy_quadratic_growth}
			c_4|\xi|^2 - C_4\leq	\Wmin^{\rm c}(\xi) \leq C_4(|\xi|^2 +1) \quad \text{ for all } \xi\in\R^3.
		\end{align}
		In particular, $\Wmin^{\rm c}: \R^3 \to [0,\infty)$ is continuous as a finite-valued convex function.
	\end{lemma}
	\begin{proof}
		Indeed, (H2), along with the observation that
		\begin{align*}
			\min_{A\in\R^{3\times 2},\, \det(\xi|A)=1}|(\xi|A)|^2  = \min_{x,y\in\R^3,\, (x\times y)\cdot \xi=1} |\xi|^2 + |x|^2+ |y|^2 = |\xi|^2 + 2 |\xi|^{-1}
		\end{align*}
		for any $\xi\in\R^3\setminus\{0\}$, allows us to infer that
		\begin{align}\label{bounds_Wbar}
			c_2|\xi|^2-C_2 \leq \Wmin(\xi) \leq C_2(|\xi|^2 + 2|\xi|^{-1} + 1) 
		\end{align}
	 	for all $\xi\in \R^3\setminus \{0\}$, which after convexification gives rise to~\eqref{convex_energy_quadratic_growth}.
 	\end{proof}
	
 	The following lemma collects a few basic properties of the zero level sets of $\Wmin$ and $\Wmin^{\rm c}$.

	\begin{remark}\label{rem:zero_level}
	Suppose that $W_0$ satifies (H1) and (H2).
	
		a) The growth assumption (H3) 
		implies immediately that $L_0(W)=L_0(W_0) = \SO(3)$.	
		
		b) If $W_0$ is frame-indifferent, i.e., $W_0(RF) = W_0(F)$ for all $F\in\R^{3\times 3}$ and any $R\in\SO(3)$, 
		then $\Wmin(\xi)$ with $\xi\in \R^3$ depends de facto only on $|\xi|$, i.e., $W(\xi)=f(|\xi|)$ for $\xi\in \R^3$ with some $f:[0, \infty)\to \R$. 
		Indeed, for any $R\in\SO(3)$ and $\xi\neq 0$,
		\begin{align*}
			\Wmin(R\xi) &=  \min\{W_0((R\xi|A)) :A \in\R^{3\times 2}, \det(R\xi|A) = 1\}\\  
					&=\min\{W_0(R(\xi|A)) : A \in\R^{3\times 2}, \det(R(\xi|A)) = 1\}\\
					&=\min\{W_0((\xi|A)) : A \in\R^{3\times 2}, \det(\xi|A) = 1\} = \Wmin(\xi). 
		\end{align*}
		In this case, $\Wmin^{\rm c} = f^{\rm c}(|\cdot|)$ and $L_0(\overline W^{\rm c}) = L_0(\overline W)^{\rm c} = \{\xi\in \R^3: |\xi|\leq \max_{t\geq 0, t\in L_0(f)} t\}$. 
	
		c) A frame-indifferent single-well energy density $W_0$ vanishing at the identity, 
		has $\SO(3)$ as zero level set of $W_0$. Hence, $L_0(\Wmin) = \{\xi\in \R^3: |\xi|=1\}$ in light of b), and after convexification, 
		\begin{align*}
			L_0(\Wmin^{\rm c}) = L_0(\Wmin)^{\rm c} = \{\xi\in \R^3: |\xi|\leq 1\}.
		\end{align*}  
	\end{remark}
	
\subsection{Technical tools}\label{subsec:tools}

	The following auxiliary result on inner perturbations is a key ingredient for the construction of recovery sequences, both in the regimes $\alpha=0$ and $\alpha\in (0,2)$, since it allows us to modify sequences subject to the incompressibility constraint in an approximate sense into ones that satisfy it exactly.
	
	Analogous techniques applicable to $3$d-$2$d dimension reduction problems were first introduced in~\cite{CoD06} and later exploited in~\cite{CoD09, ChL13}.
	We adapt the method to the $3$d-$1$d context, where perturbing one of the cross-section variables, instead of both,   is enough to realize the desired determinant condition.   
	
	\begin{lemma}[Inner perturbation]\label{lem:reparam_det=1}
		Let $\gamma>0$ and $J\subset J'\subset \R$ be bounded closed intervals such that $0\in J$ and $J$ is compactly contained in the interior of $J'$. Further, let $Q_L:=[0, L]\times J\times J\subset \R^3$ and analogously, $Q_L':=[0, L]\times J'\times J'$, and define $\Pi_3:Q_L\to J$ by $x\mapsto x_3$ for $x\in Q_L$. 
		
		If $(v_\epsilon)_{\epsilon}\subset C^2(Q_L';\R^3)$ such that
		\begin{align}\label{estimates_det}
			\|\det \rn{v} -1\|_{C^1(Q_L')}= \Ocal(\epsilon^\gamma),
		\end{align}
		then there exists a sequence $(\Phi_\eps)_{\eps} \subset C^1(Q_L;Q_L')$ with functions of the form
		\begin{align*}
			\Phi_\eps(x) = (x_1, x_2, \varphi_\eps(x)), \quad x \in Q_L,
		\end{align*}
		where $\varphi_\eps\in C^1(Q_L;J')$ for $\eps>0$ is such that
		\begin{align}\label{est_varphi}
			\|\varphi_\eps- \Pi_3\|_{C^1(Q_L;\R^3)} = \Ocal(\eps^\gamma), 
		\end{align}
		and the perturbed sequence 
		$(u_\eps)_{\eps}\subset C^1(Q_L;\R^3)$ defined by $u_\eps:=v_\eps\circ \Phi_\eps$  satisfies 
		\begin{align}\label{det=1}
			\det \rn{u} = 1 \quad\text{ everywhere in $Q_L$ for $\eps>0$.}
		\end{align}	
	\end{lemma}
	\begin{proof}
	We subdivide the construction of a sequence $(\varphi_\eps)_{\eps}\in C^1(Q_L;J')$ satisfying the conditions~\eqref{est_varphi} and~\eqref{det=1} into two steps. 
	The arguments are strongly inspired by the ideas and techniques of~\cite{CoD06, CoD09}. 
	
		Note that according to~\eqref{estimates_det}, 
	there is a constant $l>0$ such that
	\begin{align}\label{lowerbound}
		\|\det \rn{v}\|_{C^0(Q_L')} \geq l
	\end{align} 
	for all $\eps>0$ sufficiently small.

	\textit{Step~1: Implementation of the determinant constraint.} 
	Recalling the definition of the rescaled gradients in~\eqref{rescaled}, we deduce from the chain rule that $u_\eps=v_\eps\circ \Phi_\eps$ satisfies
		\begin{align*}
			\det \rn{u} &=\det (\nabla v_\epsilon \circ \Phi_\epsilon)\det \rn{\Phi}= \tfrac{1}{\epsilon^2}\det (\nabla v_\epsilon\circ\Phi_\epsilon)\det \nabla\Phi_\epsilon  \\
				&= \det (\rn{v}\circ\Phi_\epsilon)\det \nabla \Phi_\epsilon
				= \det (\rn{v}\circ\Phi_\epsilon) \partial_3\ffi_\epsilon 
		\end{align*}
	 on $Q_L$. 	
	Hence, condition~\eqref{det=1} is fulfilled if $\varphi_\eps$ solves the following initial value problem: 
	For each $x_1\in [0, L]$ and $x_2\in J$, 
	\begin{align}\label{ODE}
		\begin{split}
			\begin{cases}
				\partial_3\ffi_\epsilon(x_1,x_2,x_3) &= \displaystyle\frac{1}{\det \rn{v}(x_1,x_2,\ffi_\epsilon(x_1, x_2, x_3))} \quad \text{for $x_3\in J$, }\\
				\ffi_\epsilon(x_1,x_2,0)&= 0;
			\end{cases}
		\end{split}
	\end{align}
	notice that in view of~\eqref{lowerbound}, the denominator on the right-hand of the differential equation is in particular non-zero; also, the choice of initial conditions is indeed admissible, considering that $0\in J$.

	 The existence of a unique solution $\varphi_\eps\in C^1(Q_L;J')$ to the initial value problem in~\eqref{ODE} with continuously differentiable dependence on the parameters $x_1$ and $x_2$ follows from standard ODE theory.  More precisely,  the argument is based on Banach's fixed point theorem, see e.g.~\cite[III, \S 13, Satz II~and~IV]{Wal00}; 
	 note that in our case, the contraction may be defined on $C^0(Q_L;J')$, since \eqref{estimates_det} implies that for any $\phi\in C^0(Q_L;J')$ and $x\in Q_L$,
	 \begin{align*}
	 	\int_0^{x_3} \frac{1}{\det \rn{v}(x_1,x_2,\phi(x_1,x_2,s))} \dd s \in J',
	 \end{align*}
	 provided $\eps$ is small enough.
		
	\textit{Step 2: Estimates for $\ffi_\epsilon$.} 
	To verify~\eqref{est_varphi} for the previously constructed $\varphi_\eps$, we are going show that 
	\begin{align}\label{estimates_phiaux1}
	  \|\partial_3\ffi_\epsilon - 1\|_{C^0(Q_L)}  = \Ocal(\epsilon^\gamma),
	\end{align}
	\begin{align}\label{estimates_phiaux2}
	 \|\ffi_\epsilon - \Pi_3\|_{C^0(Q_L)} = \Ocal(\epsilon^\gamma),
	\end{align}
	\begin{align}\label{estimates_phiaux3}
		\|\partial_1 \ffi_\epsilon\|_{C^0(Q_L)}= \Ocal(\epsilon^\gamma) \quad \text{and} \quad  \|\partial_2 \ffi_\epsilon\|_{C^0(Q_L)} = \Ocal(\epsilon^\gamma).
	\end{align}
	Indeed,~\eqref{estimates_phiaux1} follows from 	
			\begin{align*}
				|\partial_3\ffi_\eps- 1| &= \frac{|\det(\rn{v}\circ \Phi_\epsilon) - 1|}{|\det (\rn{v}\circ \Phi_\epsilon)|}\leq \frac{1}{l} |\det (\rn{v}\circ \Phi_\epsilon) - 1|\leq \frac{1}{l} \|\det\rn{v} -1\|_{C^0(Q_L)} 
		\end{align*}
		in combination with~\eqref{estimates_det}, and it suffices for~\eqref{estimates_phiaux2} to observe that 
		\begin{align}\label{phi}
			|\ffi_\eps(x) - x_3| \leq \Bigl|\int_0^{x_3}\partial_3\ffi(x_1,x_2,s) - 1\dd s\Bigl| \,\leq |x_3| \|\partial_3\varphi_\eps-1\|_{C^0(Q_L)}
		\end{align}
		for any $x\in Q_L$. 
		
		For the proof of~\eqref{estimates_phiaux3}, it is convenient to rewrite~\eqref{ODE} equivalently in terms of the integral equation 
		\begin{align}\label{integral_equation}
			 \int_0^{\varphi_\eps(x)}\det \rn{v}(x_1,x_2, s)\dd s = x_3 \qquad \text{for $x\in Q_L$.}
		\end{align}
		By the Leibniz integral rule, differentiating~\eqref{integral_equation} with respect to $x_1$ leads to
		\begin{align*}
			 \int_0^{\ffi_\epsilon(x)} \partial_1 [\det \rn{v}(x_1,x_2,s)]\dd s + \partial_1\ffi_\epsilon(x) \det \rn{v}(x_1,x_2,\ffi_\epsilon(x)) =0
		\end{align*}
		for $x\in Q_L$, and hence, along with~\eqref{lowerbound},
		\begin{align*}
			|\partial_1\ffi_\epsilon(x)| &\leq \frac{1}{l} \Bigl| \int_0^{\ffi_\epsilon(x)} \partial_1 \det \rn{v}(x_1,x_2,s) \dd s \Bigr| \leq \frac{|\varphi_\eps(x)|}{l}\|\nabla(\det \rn{v})\|_{C^o(Q_L)}. 
		\end{align*}
		In view of~\eqref{phi} and~\eqref{estimates_det},
		this gives the first part of~\eqref{estimates_phiaux3}. The second part involving $\partial_2\varphi_\eps$ follows in the same way.
	\end{proof}

	The next lemma provides a technical tool that, intuitively speaking, allows us to round off the corners of a piecewise affine curve in such a way that the resulting mollification is a regular curve and still piecewise affine on most of its domain. 
	This can be achieved with the help of B\'ezier curves, see e.g.~\cite{Far02}. 	 
	Although there is a substantial literature on the subject, we have not been able to track down the specific statement needed for the construction of a recovery sequence in Theorem~\ref{theo:strings=0}. We present a self-contained proof in the appendix.
	
	\begin{lemma}[Mollification via B\'ezier curves]\label{prop:bezier_mollifying}
		Let $k\in \N$ and $u\in \Affpw$. Further, let $\Gamma$ be the finite set of points in $(0, L)$ where $u$ is not differentiable, and suppose that $u'\neq 0$ a.e.~in $(0, L)$. 
		
		 Then there exists a sequence $(u_i)_{i}\subset C^k([0,L];\R^3)$ with the following three properties: 
		\begin{itemize}
		\item[$(i)$] (Uniform bounds) $c\leq |u_i'|\leq C$ in $[0,L]$ for all $i\in \N$ with constants $c,C>0$ depending only on $u$;
		\item[$(ii)$] (Affinity) $u_i = u \text{ on } [0,L]\setminus \Gamma_i$ 
		with $\Gamma_i = \{t\in [0,L] : \dist(t,\Gamma) \leq \tfrac{1}{i}\}$ for $i\in \N$ large enough;
		\item[$(iii)$] (Convergence) $u_i \to u$ in $H^{1}(0,L;\R^3)$ as $i\to \infty$.
		\end{itemize}
	\end{lemma} 	

\section{The regime $\alpha = 0$}\label{sec:alpha=0} 
	The first main result derives an effective one-dimensional model for incompressible elastic strings via a $\Gamma$-convergence analysis of the energies $\Ical_\eps^\alpha$ with $\alpha=0$ in the limit of vanishing $\eps$. 
		Its two main characteristics reflect an optimization over all deformations of the cross section at finite thickness and a relaxation procedure minimizing the energy over possible microstructures.
	\begin{theorem}\label{theo:strings=0}
		For $\eps>0$, let $\Ical_\eps^0$ as in~\eqref{energy_all_scaling_regimes}, assuming that $W_0$ satisfies (H1) and (H2).
		Then, $(\Ical_\epsilon^0)_{\epsilon}$ $\Gamma$-converges regarding the weak topology in $H^1(\Omega;\R^3)$ to
		\begin{align*}
			\Ical^0:  H^1(\Omega;\R^3)&\to [0, \infty],\quad 
				u\mapsto \begin{cases}\displaystyle  |\omega| \int_0^L \Wmin^{\rm c}(u') \dd x_1 &\text{ if } u\in H^1(0,L;\R^3),\\ \infty &\text{ otherwise,}\end{cases}
		\end{align*}
		with $\Wmin:\R^{3}\to [0,\infty)$ as defined in~\eqref{overlineW}. 
					
	Furthermore, every sequence $(u_\eps)_\eps\subset H^1(\Omega;\R^3)$ with $\int_\Omega u_\eps \dd{x}=0$ and $\sup_{\eps>0}\Ical_\eps^0(u_\eps) <\infty$ is relatively weakly compact.
	\end{theorem}
	\begin{proof} 
		The overall idea of the proof follows along the lines of~\cite{CoD06}, but the arguments need to be suitably modified for this setting of $3$d-$1$d reduction. 
		This involves a taylored mollification and frame construction for piecewise affine regular curves, as well as elements from~\cite{ABP91}, see also \cite{Sca06}. 
		The key ingredient for realizing the volumetric constraint in the construction of the recovery sequence is Lemma~\ref{lem:reparam_det=1}. 
	
		\textit{Part I: Lower bound and compactness.} Let $(u_\epsilon)_{\eps}\subset H^1(\Omega;\R^3)$ be a sequence of functions with zero mean value and uniformly bounded energy regarding $(\Ical_\eps^0)_{\eps}$.
		Due to $W\geq W_0$ and the fact that, by assumption, $W_0$ satisfies the necessary properties to apply the compactness of the corresponding compressible problem (see e.g.~\cite[Theorem~2.1]{ABP91}), we conclude the existence of a subsequence of $(u_\eps)_\eps$ converging weakly in $H^1(\Omega;\R^3)$ to some $u\in H^1(0,L;\R^3)$.
		Since $\Wmin^{\rm c}$ is convex, continuous and bounded from below, the functional
		\begin{align*}
			L^2(\Omega;\R^3)\ni v \mapsto \int_\Omega \Wmin^{\rm c}(v) \dd x 
		\end{align*}
		is $L^2$-weakly lower semi-continuous (see e.g.~\cite[Theorem~1.3]{Dac08}), and hence, 
		\begin{align*}
			\liminf_{\epsilon\to 0} \Ical_\epsilon^0(u_\epsilon) & = \lim_{\epsilon\to 0} \int_\Omega W(\rn{u}) \dd x 
					\geq \liminf_{\epsilon\to 0} \int_\Omega \Wmin^{\rm c}(\partial_1 u_\epsilon) \dd x \\ 
				& \geq \int_\Omega \Wmin^{\rm c}(\partial_1 u) \dd x = |\omega| \int_0^L \Wmin^{\rm c}(u') \dd x_1,
		\end{align*}
		which is the desired liminf-inequality. 	
		
		\textit{Part~II: Upper bound. } Let $u\in H^1(0, L;\R^3)$. For easier reading, we subdivide the argument into several steps. 

		\textit{Step 1: Affine approximation.} Considering that $\Affpw$ is dense in $H^1(0,L;\R^3)$, 
		one can find a sequence $(\tilde v_j)_{j}\subset \Affpw$ such that 
		\begin{align}\label{124}
		\tilde v_j\to u\quad\text{ in $ H^1(0,L;\R^3)$.}
		\end{align} 
		Then, in light of the continuity of $\Wmin^{\rm c}$ and its quadratic growth \eqref{convex_energy_quadratic_growth}, we can pass to the limit via the Vitali-Lebesgue convergence theorem to obtain
		\begin{align}\label{125}
			\lim_{j\to \infty} \int_0^L \Wmin^{\rm c}(\tilde v_j') \dd x_1 = \int_0^L \Wmin^{\rm c}(u') \dd x_1.
		\end{align}
		
		\textit{Step 2: Relaxation.} Next, we will construct a sequence $(v_j)_{j}\subset \Affpw$ such that $v_j\weakly u$ in $ H^1(0,L;\R^3)$ and
		\begin{align}\label{126}
			\lim_{j\to \infty} \int_0^L \Wmin(v_j') \dd x_1 \leq \int_0^L \Wmin^{\rm c}(u') \dd x_1;
		\end{align}	
		in particular, this means that $v_j'\neq 0$ a.e.~in $(0,L)$ for $j\in \mathbb N$. 
		
		For $j\in \mathbb N$, let $\tilde v_j$ be the function from Step 1, and denote the finitely many disjoint open subintervals of $(0, L)$ on which $\tilde v_j'$ is constant by $\tilde I\ui{n}_j$ with $n=1,...,\tilde N_j$; notice that 
		\begin{align*}
		\textstyle \bigl|(0, L)\setminus \bigcup_{n=1}^{\tilde N_j} \tilde I_j^{(n)}\bigr|=0.
		\end{align*} 
		The idea is to modify $\tilde v_j$ suitably on each $\tilde I\ui{n}_j$. 
		To be precise, by classical arguments in convex analysis and relaxation theory 
 (see~e.g.~\cite[Theorem~2.35]{Dac08}), one can find functions $\phi\ui{n}_j\in H^1_0(\tilde I\ui{n}_j;\R^3)\cap \Affpw$, such that
		\begin{align*}
			\dashint_{\tilde I\ui{n}_j} \Wmin(\tilde v_j' + (\phi\ui{n}_j)') \dd x_1 &\leq \Wmin^{\rm c}(\tilde v_j') + \frac{1}{j}
			\end{align*}	
		and
			\begin{align}\label{127}
			\int_{\tilde I\ui{n}_j} |\phi\ui{n}_j|^2 \dd x_1 &\leq \frac{1}{j^2} |\tilde I\ui{n}_j|;
		\end{align}
		the add-on~\eqref{127} follows via a simple refinement argument, just replace $\phi_j\ui{n}$ with a piecewise affine function that consists of multiple scaled copies of the latter. 
				
		Define $v_j := \tilde v_j + \sum_{n=1}^{\tilde N_j} \phi\ui{n}_j \mathbbm{1}_{\tilde I\ui{n}_j}$, where  $\mathbbm{1}_I$ for some $I\subset \R$ denotes the associated indicator function with values $0$ and $1$. Then, $\norm{v_j-\tilde v_j}_{L^2(0,L;\R^3)}\leq \frac{\sqrt{L}}{j}$ and
		\begin{align*}
		\int_0^L \Wmin(v_j') \dd x_1 &
		= \sum_{n=1}^{\tilde N_j} \int_{\tilde I\ui{n}_j} \Wmin(\tilde v_j' +(\phi\ui{n}_j)')  \dd x_1\\
		&\leq \sum_{n=1}^{\tilde N_j}(\Wmin^{\rm c}(\tilde v_j') + \tfrac{1}{j})|\tilde I\ui{n}_j| =  \int_0^L \Wmin^{\rm c}(\tilde v_j') \dd x_1 + \frac{L}{j}.
		\end{align*}
	Hence, letting $j\to \infty$ implies~\eqref{126} in view of~\eqref{125}, as well as $v_j\to u$ in $L^2(0,L;\R^3)$ due to~\eqref{124}. 
		The observation that $(v_j)_{j}$ is uniformly bounded in $ H^1(0,L;\R^3)$ as a consequence of the lower bound in~\eqref{bounds_Wbar},
		allows us to conclude the desired weak convergence $v_j\weakly u$ in $H^1(0,L;\R^3)$.
		
		\textit{Step 3: Mollification of the piecewise affine approximations.} 
		With the help of Lemma \ref{prop:bezier_mollifying} applied to each $v_j$ from Step~2 and a diagonalization argument,
		we obtain a sequence of functions $(u_j)_{j}\subset C^3([0,L];\R^3)$ with these properties:
		\begin{itemize}
		\item[$(i)$] for every $j\in \N$ there are $0<l_j\leq L_j$ such that
		$l_j\leq |u_j'|\leq L_j$ in $[0, L]$; without loss of generality, we may assume that $l_j\leq 1$; 
		\item[$(ii)$] for every $j\in \N$ there are finitely many disjoint open intervals $I_j\ui{n}\subset [0, L]$ with $n=1,...,N_j$ such that the restriction of $u_j$ to $I_j\ui{n}$ is affine and coincides with $v_j|_{I_j\ui{n}}$;
		\item[$(iii)$]  $\lim_{j\to\infty} |\Gamma_j| (L_j^2 + l_j^{-2} +1) = 0$, where $\Gamma_j: =[0,L]\setminus \bigcup_{n=1}^{N_j} I_j\ui{n}$ for $j\in \N$;
		\item[$(iv)$]  $u_j-v_j\to 0$ in $H^1(0,L;\R^3)$, and hence by Step~2, 
		\begin{align}\label{convergenceH1}
			u_j\weakly u \quad \text{in $H^1(0,L;\R^3)$.}
		\end{align}
		\end{itemize}
			
		\textit{Step 4: Taylored frame.} For any curve $u_j$ with $j\in \N$ as in the previous step, let $n_j\in C^2([0, L];\R^3)$ be a normal unit vector field  along $u_j$, meaning $u_j'\cdot n_j=0$ and $|n_j|=1$ everywhere in $[0, L]$; we may assume without restriction that $n_j$ is constant whenever $u_j'$ is. 
	Moreover, define
		\begin{align}\label{frame}
			b_j := \frac{u_j' \times n_j}{|u_j'\times n_j|^2} \in C^2([0,L];\R^3);
		\end{align}  
		indeed, the denominator in~\eqref{frame} is non-zero, because $n_j$ is orthogonal to $u_j'$ and $u_j$ a regular curve by Step~3\,$(i)$.
		By definition, the triple $(u_j', n_j, b_j)$ forms an orthogonal moving frame along the trajectory given by $u_j$. Our aim in this step is to modify this moving frame into a version that is well-suited for the construction of an approximating sequence for $u$  along which the energies converge as well, cf.~Step~5. 
		
		To this end, recall that $u_j$ is affine on each $I_j\ui{n}$ with $n=1, \ldots, N_j$ according to Step~3\,$(ii)$, that is, 
		\begin{align}\label{xinj}
		u_j'|_{I\ui{n}_j}=\xi_j\ui{n}\qquad \text{with $\xi_j\ui{n}\in \R^3$.}
		\end{align} Let $A_j\ui{n} \in\R^{3\times 2}$ be such that
		\begin{align}\label{Ajn}
		W\bigl((\xi_j\ui{n} | A_j\ui{n})\bigl) =\min_{A\in \R^{3\times 2}} W\bigl((\xi_j\ui{n}|A)\bigl)= \Wmin(\xi_j\ui{n});
		\end{align}
		in particular, $\det (\xi_j\ui{n}|A_j\ui{n})=1$, cf.~\eqref{overlineW}.

		Moreover, for fixed $\delta, \eta>0$ sufficiently small, we consider compactly contained nested open subintervals $I_{j, \delta, \eta}\ui{n} \subset I_{j,\delta}\ui{n}\subset I_j\ui{n}$ with 
		\begin{align}\label{measure_deltaeta}
		|I_j\ui{n}\setminus I_{j, \delta}\ui{n}| \leq \delta\quad \text{ and } \quad |I_{j, \delta, \eta}\ui{n}\setminus I_{j,\delta}\ui{n}| \leq \eta.
		\end{align}
		
		Based on these definitions, we find $\bar n_{j, \delta} \in C^2([0,L];\R^3)$ with the properties that 
				\begin{align*}
		\bar n_{j, \delta} \restrict{\Gamma_j} = n_j	\quad \text{and} \quad \bar n_{j, \delta}\restrict{I\ui{n}_{j,\delta}} = A_j^{(n)} e_1\text{ for $n=1, \ldots, N_j$,}
		\end{align*}	
		as well as 
		\begin{align}\label{notvanishing}
		|\bar n_{j, \delta}| <R_j \quad \text{and}\quad |u_j'\times \bar{n}_{j,\delta}| > r_j \text{ in $[0,L]$,}
		\end{align}
		where $R_j:= 2\max\{1, \max_{j=1, \ldots, N_j} |A_j\ui{n}e_1|\}$ and 
		\begin{align*}
		r_j := \frac{1}{2}\min\{l_j, \min_{n=1, \ldots, N_j} |\xi_j\ui{n} \times A_j\ui{n}e_1|\}>0.
		\end{align*} 
	
Geometrically speaking, we choose the free curve segments of $\bar n_{j, \delta}$ on $I_j\ui{n}\setminus I_{j, \delta}\ui{n}$ in such a way that $\bar{n}_{j, \delta}$ lies within the ball around the origin of radius $2\max\{1, |A_j\ui{n}e_1|\}$, but in the complement of the cylinder centered in the origin with axis pointing in the direction of $\xi_j\ui{n}$ and circular cross section of radius $\frac{1}{2}\min\{\big|\tfrac{\xi_j\ui{n}}{|\xi_j\ui{n}|}\times A_j\ui{n}e_1\big|,1\}$, 
Such a choice is possible, because the selected two radii guarantee that both the value of $n_j$ on $I_j\ui{n}$ and $A_j^{(n)}e_1$ are contained in the specified path-connected region, see Figure~\ref{fig:cylinder} for illustration. 
	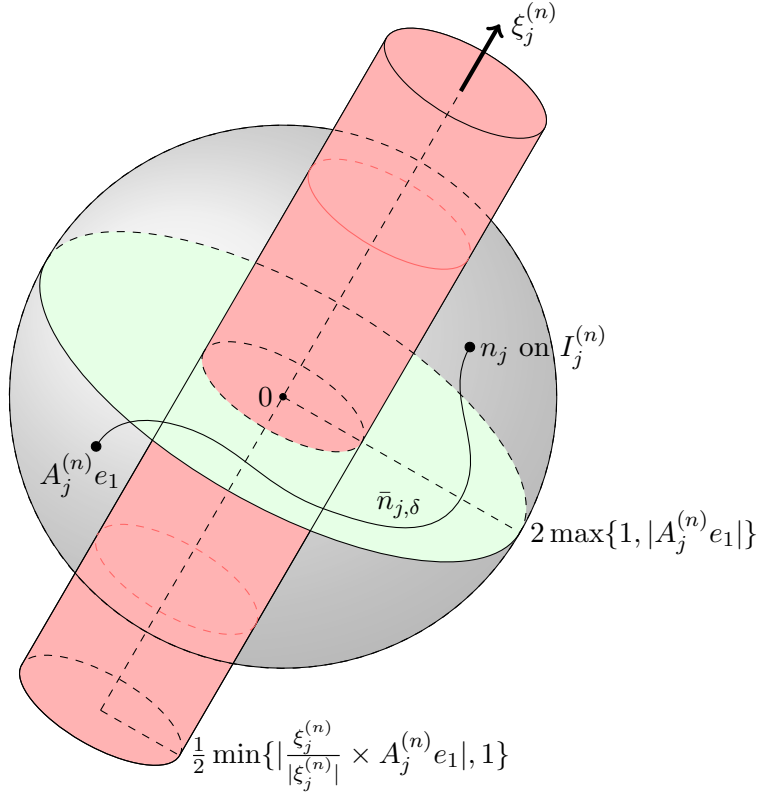
\begin{figure}[h]
		\centering
		\begin{tikzpicture}[scale = 0.6]
			\begin{scope}[rotate=-30]
				\shade[name path = ball, ball color = gray!40, opacity = 0.4] (0,0) circle (6cm); 
			  	\draw [fill=red!30,name path = cylinder](-2,-8)--(-2,8) arc (-180:-360:2 and 0.8) --++(0,-16) arc (0:180:2 and -0.8); 
			  	\fill [name path = ellipse, dashed, green!10] (-6,0) arc (-180:0:6 and -2.4) arc (0:-180:6 and 2.4); 
			  	\fill[red!30] (-2,0)--(-2,8) arc (-180:-360:2 and 0.8) --++(0,-8) arc (0:180:2 and -0.8); 
			  	\draw (-2,-8)--(-2,8) arc (-180:-360:2 and 0.8) --++(0,-16) arc (0:180:2 and -0.8); 
			  	\draw [dashed] (0,0) circle (6cm); 
			 	\fill [fill=black] (0,0) circle (0.5mm); 
				\draw[dashed] (0,-8) -- (2,-8) node[anchor=west]{$\frac{1}{2}\min\{|\tfrac{\xi_j^{(n)}}{|\xi_j\ui{n}|}\times A_j^{(n)}e_1 |, 1\}$}; 
				
				\draw[dashed] (2,-8) arc (0:180:2 and 0.8);
				\draw (2,8) arc (-0:-180:2 and 0.8);
				
				\draw [dashed](0,-8)--(0,8); 
				\draw [ultra thick, ->] (0,7.8)--(0,9.5) node[anchor = west] {$\xi_j^{(n)}$};
				
				\path [name intersections={of=ball and cylinder}];
				\draw (intersection-1) arc (70.54:-250.46:6);
					
				\draw [red!60]($(intersection-2)+(0,-1)$) arc (180:0:2 and -0.8);
				\draw [red!60, dashed]($(intersection-2)+(0,-1)$) arc (-180:0:2 and -0.8);
				
				\filldraw  (-3,-3) circle  (3pt);
				\draw (-3.35,-3.1) node [anchor = north]  {$A_j\ui{n} e_1$};
				\filldraw  (3,3) circle  (3pt) node[anchor=west] {$n_j$ on $I_j\ui{n}$}; 
				\filldraw  (0,0) circle  (2pt) node[anchor=east] {$0$};
				\draw (-3,-3) to [out=90,in=200] (-2,-1.8) to [out = 20, in= 190] (2,-1.7) to [out=10, in=270] (4.5,0) to [out=90,in=270] (3,3);
				\draw (3.4,-0.8) node {$\bar{n}_{j,\delta}$};

				\draw [red!60,dashed]($(intersection-4)+(0,1)$) arc (0:180:2 and -0.8);
				\draw [red!60, dashed]($(intersection-4)+(0,1)$) arc (0:-180:2 and -0.8);

				\draw [dashed] (-6,0) arc (-180:0:6 and -2.4) ;
				\draw (6,0) arc (0:-180:6 and 2.4);
				
				\draw [dashed] (0,0) -- (6,0) node[anchor=west] {$2\max\{1,|A_j^{(n)}e_1|\}$};
				
				\draw [dashed](-2,0) arc (180:0:2 and -0.8) arc (0:180:2 and 0.8);
			\end{scope}
		\end{tikzpicture}
		\caption{Sketch of the construction of $\bar{n}_{j, \delta}$ on $I_j\ui{n}\setminus I_{j, \delta}\ui{n}$.}
		\label{fig:cylinder}
	\end{figure}
	
		As a customized replacement for $b_j$ from~\eqref{frame}, we introduce $\bar b_{j, \delta, \eta} \in C^2([0, L];\R^3)$ given by 
		\begin{align*}
			\bar b_{j, \delta, \eta} =\bar b_{j, \delta}+ \sum_{n=1}^{N_j} \psi_{j, \delta, \eta}\ui{n}(A_j\ui{n} e_2-\hat b_j\ui{n})
		\end{align*}
		where $\psi_{j, \delta, \eta}\ui{n}:[0,L]\to [0,1]$ are smooth cut-off functions with compact support in $I_{j, \delta}\ui{n}$ satisfying $\psi_{j, \delta, \eta}\ui{n}= 1$ on $I_{j,\delta,\eta}\ui{n}$, 
		\begin{align*}		
			\bar{b}_{j, \delta} := \frac{u_j'\times \bar{n}_{j, \delta}}{|u_j'\times \bar{n}_{j, \delta}|^2}  \quad \text{and}\quad \hat b_j\ui{n} := \frac{\xi_j\ui{n}\times A_j\ui{n}e_1}{|\xi_j\ui{n}\times A_j\ui{n} e_1|^2};
		\end{align*} 
		we remark that the last two quanities are well-defined due to~\eqref{notvanishing} and~the fact that 
		\begin{align*}
		\det(\xi_j\ui{n}|A_j\ui{n})=(\xi_j\ui{n}\times A_j\ui{n}e_1)\cdot A_j\ui{n}e_2\neq 0.
		\end{align*} 
		
	Next, we collect a few useful properties of the newly constructed moving frames $(u_j', \bar n_{j, \delta}, \bar b_{j, \delta, \eta})$. 		
 	Setting
		\begin{align*}
			F_{j, \delta, \eta}:= (\bar n_{j, \delta} | \bar b_{j, \delta, \eta}) \in C^2([0,L];\R^{3\times 2}), 
		\end{align*}
		we observe that
		\begin{align}\label{det1}
			\det(u_j'|F_{j, \delta, \eta}) &=  \det(u_j'| \bar n_{j, \delta}| \bar b_{j, \delta})+ \sum_{n=1}^{N_j} \psi_{j, \delta, \eta}\ui{n}  \big[ \det(\xi\ui{n}_j| A_j\ui{n}) -\det(\xi_j\ui{n}|A_j\ui{n}e_1|\hat b_j\ui{n}) \big]\\ &= \det(u_j'| \bar n_{j, \delta}| \bar b_{j, \delta})  = 1,\nonumber
		\end{align}
		and
		\begin{align}\label{estF}
		|F_{j, \delta, \eta}|^2 =  |\bar n_{j, \delta}|^2 + |\bar b_{j, \delta, \eta}|^2 
		   \leq C (L_j^2 + l_j^{-2} +1)
		\end{align}
		with a constant $C>0$ independent of $j, \delta$ and $\eta$; see again Step~3, where the constants $L_j$ and $l_j$ have been introduced. 
		Indeed, to see~\eqref{estF}, we infer from~\eqref{notvanishing} together with the estimate 
		\begin{align}\label{njdelta}
		|\xi_j\ui{n}\times A_j\ui{n}e_1|\,|A_j\ui{n}e_2| \geq |\det(\xi_j|A_j\ui{n}e_1|A_j\ui{n}e_2)|= 1
		\end{align} 
		that
		\begin{align}\label{bjdelta}
		|\bar n_{j, \delta}| & \leq R_j \leq 2(1+  \max_{n=1, \ldots, N_j} |A_j\ui{n}e_1| ),
		\end{align}
		and 
\begin{align*}
		|\bar b_{j, \delta, \eta}|&\leq |u_j'\times \bar n_{j, \delta}|^{-1} + \max_{n=1, \ldots, N_j} (|A_j\ui{n}e_2| + |\xi_j\ui{n}\times A_j\ui{n}e_1|^{-1})
		\leq r_j^{-1}  +2  \max_{n=1, \ldots, N_j} |A_j\ui{n}e_2| \\ & \leq 2(l_j^{-1}  +2  \max_{n=1, \ldots, N_j} |A_j\ui{n}e_2|);
		\end{align*}	
		 in the last inequality, we have used in particular that 
		 \begin{align*}
		 r_j\geq \frac{1}{2}\min\{l_j, \min_{n=1,\ldots,N_j}|A_j\ui{n}e_2|^{-1}\}=\frac{1}{2}\min\{l_j, (\max_{n=1,\ldots,N_j}|A_j\ui{n}e_2|)^{-1}\}
		 \end{align*} 
		 due to~\eqref{njdelta}.
		
		Due to~\eqref{Ajn} and the growth properties of $\overline W$ and~$W_0$ from~\eqref{bounds_Wbar} and (H2), there exists a constant $C>0$ such that
		\begin{align}\label{Anj2} |A_j\ui{n}|^2 \leq C(|\xi_j\ui{n}|^2 + |\xi_j\ui{n}|^{-1}+ 1)
		\end{align} for all $n=1, \ldots, N_j$ and $j\in \N$.
		Hence, in view of~\eqref{xinj} and~Step~3\,$(i)$, combining~\eqref{Anj2} with~\eqref{njdelta} and~\eqref{bjdelta} 
		eventually implies~\eqref{estF}.  
		
		\textit{Step 5: Recovery sequence.}
		 Let $J\subset J'\subset \R$ as in Lemma~\ref{lem:reparam_det=1} and $\overline{\omega}\subset J\times J$. We start by considering for fixed $j\in \N$ and $\delta, \eta>0$ the auxiliary sequence $(v_{j, \delta, \eta, \eps})_{\epsilon}\subset C^2(Q_L';\R^3)$ given by 
		\begin{align}\label{recovery_alpha0}
			v_{j,\delta, \eta, \eps}(x) := u_j(x_1) + \epsilon x_2 \bar n_{j, \delta}(x_1) + \epsilon x_3 \bar b_{j, \delta, \eta}(x_1) \quad \text{for $x\in Q_L'$,}
		\end{align}
		 cf.~e.g.~\cite[Proposition 3.3]{ABP91}. 
		 Clearly, 
		\begin{align}\label{C1}
			\|v_{j, \delta, \eta, \eps} -u_{j}\|_{C^1(\overline \Omega;\R^3)}\to 0\quad \text{ as $\eps\to 0$. }
		\end{align}
		Using~\eqref{det1} and the uniform boundedness of $\bar n_{j, \delta}$, $\bar b_{j, \delta, \eta}$ and their derivatives uniformly in $[0,L]$, we obtain for the following terms involving the rescaled gradients of $v_{j, \delta, \eta, \eps}$ that 
		\begin{align}\label{det2}
			\|\det \nabla^\epsilon v_{j, \delta, \eta, \eps}-1\|_{C^1(\overline \Omega)} = \Ocal(\epsilon),
		\end{align} 
		and
		\begin{align}\label{W0convergence}
		\|W_0(\nabla^\eps v_{ j, \delta, \eta, \eps}) - W_{0}\bigl((u_j'|F_{j, \delta, \eta})\bigr)\|_{C^0(\overline \Omega)} \to 0\quad \text{as $\eps\to 0$.}
		\end{align} 
			
		As a consequence of the choices in Step~4, we find that $(u_j'|F_{j, \delta, \eta})= (\xi_j\ui{n}|A_j\ui{n})$ on $I_{j, \delta,\eta}\ui{n}$ for all $n=1, \ldots, N_j$ and $\eps>0$. Hence, along with
		~\eqref{overlineW}, \eqref{det1},~\eqref{Ajn} and (H2),		
		\begin{align}\label{est129} 
	  \int_{\Omega} W_{0}\bigl((u_j'|F_{j, \delta, \eta})\bigr) \dd{x} 
	& \leq  |\omega| \sum_{n=1}^{N_j} \overline{W}(\xi_j\ui{n}) |I_{j, \delta, \eta}\ui{n}| 
	 +  C_2 |\omega|\,\int_{[0, L]\setminus \bigcup_{n=1}^{N_j} I_{j,\delta, \eta}\ui{n}} |u_j'|^2 + |F_{j, \delta, \eta}|^2+1 \dd{x_1}
	  \nonumber \\
			& \leq |\omega|\int_0^L \overline{W}(v_j')\dd{x_1} + C |\omega| (|\Gamma_j|+N_j\delta\eta) (L_j^2 + l_j^{-1} +1)
		\end{align}
		with $C>0$ independent of $ j, \delta, \eta$ and $\eps$; \color{black}
		in the last estimate, we have exploited~\eqref{xinj} in combination with Step~3\,$(ii)$, as well as~\eqref{estF} and~\eqref{measure_deltaeta}. 
	
		What prevents a suitably diagonalized version of $(v_{j, \delta, \eta, \eps})_{\eps}$ from being a valid recovery sequence is its failure to satisfy the incompressibility constraint. This issue can be overcome by modifying the sequence according to Lemma~\ref{lem:reparam_det=1}, which is applicable due to~\eqref{det2}. Precisely, 
		one obtains
		$(u_{j, \delta, \eta, \eps})_{\epsilon}\subset C^1(\overline{\Omega};\R^3)$ such that 
		$\det \nabla^\epsilon u_{j, \delta, \eta, \eps} = 1$ for every $\epsilon>0$ and 
		\begin{align}\label{C2}
			\|u_{j, \delta, \eta, \eps} -  v_{j, \delta, \eta, \eps}\|_{C^1(\overline \Omega;\R^3)} = \Ocal(\eps^2);
		\end{align}
		the latter follows from~\eqref{est_varphi} in combination with the special structure of $v_{j, \delta, \eta,\eps}$ in~\eqref{recovery_alpha0}.  
	
		Hence, $u_{j, \delta, \eta, \eps}\to u_j$ uniformly on $\overline \Omega$ and
		\begin{align}\label{comparison}
			\Ical_\eps^0(u_{j, \delta, \eta, \eps}) - \int_{\Omega} W_0(\nabla^\eps v_{j, \delta, \eta, \eps}) \dd{x} = \int_{\Omega} W_0(\nabla^\eps u_{j, \delta, \eta, \eps}) \dd{x} - \int_{\Omega} W_0(\nabla^\eps v_{j, \delta, \eta, \eps}) \dd{x} \to  0
		\end{align}
		as $\eps\to 0$.
		
   		Joining~\eqref{W0convergence} and~\eqref{est129} with~\eqref{comparison}, under consideration of~\eqref{126},~\eqref{measure_deltaeta} and Step~3\,$(iii)$, gives
		\begin{align*}
			\limsup_{j\to\infty}\limsup_{\delta\to 0}\limsup_{\eta\to 0}\limsup_{\epsilon\to 0} \Ical_\epsilon^0(u_{j, \delta, \eta, \eps})\leq |\omega|\int_0^L\overline{W}^{\rm c}(u')\dd{x_1}= \Ical^0(u).
		\end{align*}
 		 Together with~\eqref{convergenceH1},~\eqref{C1} and~\eqref{C2}, we can finally extract a diagonal sequence $(u_\eps)_{\eps}$ in the sense of Attouch~\cite[Lemma~1.15, Corollary~1.16]{Att84} such that 
		\begin{align*}
			\limsup_{\eps\to 0} \Ical_\epsilon^0(u_\eps)\leq \Ical^0(u)
		\end{align*}
		and  $u_\eps\weakly u$ in $H^1(\Omega;\R^3)$,  which concludes the proof.
	\end{proof}

\section{The regime $0<\alpha < 2$}\label{sec:alpha>0}
	In the intermediate scaling regime $\alpha\in(0,2)$, all admissible deformations for the one-dimen\-sional limit model can be realized with zero energy, 
	as our next theorem shows. 
	\begin{theorem}\label{theo:strings>0}
		For $\eps>0$, let $\Ical_\eps^\alpha$ with $0<\alpha<2$ as in~\eqref{energy_all_scaling_regimes} such that $W_0$ satisfies (H1), as well as (H2) if $\alpha<\frac{1}{2}$, and (H3) if $\alpha\geq \frac{1}{2}$.  
		
		Then, $(\Ical_\epsilon^\alpha)_{\epsilon}$ $\Gamma$-converges with respect to the weak topology in $H^1(\Omega;\R^3)$ to
		\begin{align*}
			\Ical^\alpha:  H^1(\Omega;\R^3)&\to [0, \infty],\quad 
				u\mapsto \begin{cases}\displaystyle 0 &\text{ if } u\in H^1(0,L;\R^3) \text{ with } u'\in L_0(\overline{W})^{\rm c} \text{ a.e.~in $(0, L)$},\\ \infty &\text{ otherwise,} \end{cases}
		\end{align*}
		where $L_0(\overline W)$ is the zero level set of $\overline W$ as in~\eqref{overlineW}. 
		
		Moreover, any $(u_\eps)_{\eps}\subset H^1(\Omega;\R^3)$ with $\int_\Omega u_\eps \dd{x}=0$ for all $\eps>0$ and $\sup_{\eps>0} \Ical_\eps^\alpha(u_\eps) <\infty$ is relatively weakly compact. 
	\end{theorem} 
	\begin{remark}\label{lem:trivial_ energ_on_ball} 
		a) Trivially, if the zero level set of $\overline W$ is empty, which is the case when $L_0(W_0)\cap \Sl(3)=\emptyset$, 
		the limit functional $\Ical^\alpha$ takes the value $\infty$ everywhere. 
		
		 b) Recall Remark \ref{rem:zero_level} c), which shows that if $W_0$ is frame-indifferent and has a single-well  energy at $\SO(3)$, then,
		\begin{align*}
				\Wmin^{\rm c}(\xi) = 0 \text{ if and only if }|\xi|\leq 1.
		\end{align*} 
		The interpretation of Theorem~\ref{theo:strings>0} in this case is that no energy is required to compress the one-dimensional limit object.
		Stretching, on the other hand, has infinite energetic cost and is therefore forbidden.
		It is interesting to observe that a comparison with~\cite[Theorem~4.5]{Sca06}, where no incompressibility constraint is imposed, yields no difference for the resulting string models.
		
	 c) By a slight adaptation (in fact, a simplification) of the proof below, the statement of Theorem~\ref{theo:strings>0} remains true if $W$ is replaced with a continuous density $W_0$ that satisfies the growth assumption (H2) if $\alpha<1$ and (H3) if $\alpha\geq 1$. This observation allows to weaken the hypotheses on the energy densities in~\cite[Theorem~4.5]{Sca06}, where $3$d-$1$d dimension reduction is performed in the unconstrained case, for $\alpha<1$; in particular, the result becomes applicable for energies of  multi-well type.  
	\end{remark}
	
	\begin{proof}[Proof of Theorem \ref{theo:strings>0}]
		Under consideration of Lemma~\ref{lem:reparam_det=1}, the proof of the upper bound comes down to a  modification and generalization of the construction in~\cite[Theorem 4.5]{Sca06}. 
		The compactness and lower bound follow as an immediate consequence of the respective results for the case $\alpha=0$. 
	
		\textit{Part~I: Lower bound and compactness.} 
		Let $(u_\epsilon)_{\epsilon} \subset H^1(\Omega;\R^3)$ be a sequence of functions with vanishing mean value. If $(u_\eps)_\eps$ has uniformly bounded energy, then there is a constant $C>0$ such that
		\begin{align*}
			\Ical_\epsilon^0 (u_\epsilon) \leq C\epsilon^\alpha
		\end{align*}
		for all $\eps>0$.
		By Theorem~\ref{theo:strings=0}, a subsequence of $(u_\epsilon)_{\eps}$ (not relabeled) converges weakly to some one-dimensional function $u\in H^1(0,L;\R^3)$ in $H^1(\Omega;\R^3)$, and 
		\begin{align*}
			0 = \liminf_{\epsilon\to 0} \Ical_\epsilon^0(u_\epsilon) \geq  |\omega|\int_0^L \Wmin^{\rm c}(u')\dd x_1.
		\end{align*}
		Thus, $u' \in L_0(\overline{W}^{\rm c})=L_0(\overline{W})^{\rm c}$ in $(0, L)$, which concludes the first part of the proof.

		\textit{Part~II: Upper bound.} Let $u\in H^1(0, L;\R^3)$ and assume that $u'\in L_0(\overline{W})^{\rm c}$ a.e.~in $(0, L)$, otherwise there is nothing to prove. We proceed in two steps.

		\textit{Step 1: Basic construction for piecewise affine functions.} 	We address first the special case when $u\in A_{\rm pw}(0,L;\R^3)$ and $u'\in L_0(\overline W)$ a.e.~in $(0, L)$. 
		
		Let
		\begin{align*}
			0 =  t\ui{0} < t\ui{1} <  \ldots < t\ui{N-1} < t\ui{N}=L
		\end{align*} 
		be a partition of the interval such that the restrictions of $u'$ to the intervals $(t\ui{n-1}, t\ui{n})$  
	 	are constant, with values $\xi\ui{n}\in L_0(\overline W)$, 
	 respectively. 
		For any $n=1, \ldots, N$, we select $A\ui{n}\in \R^{3\times 2}$ to be a solution to the minimization problem  in~\eqref{overlineW} defining~$\overline{W}(\xi\ui{n})$, so that
		\begin{align}\label{139}
			\det (\xi\ui{n}|A\ui{n}) =1 \quad \text{and}\quad W\bigl((\xi\ui{n}|A\ui{n})\bigr)= \overline W(\xi\ui{n}) =0. 
		\end{align}

		Next, we set $\Mcal :=\SO(3)$ if $\alpha\geq\frac{1}{2}$ and $\Mcal:=\Sl(3)$ if $\alpha<\frac{1}{2}$, and exploit the fact that $\Mcal$ is a path-connected smooth manifold 
		to obtain 
		\begin{align*}
			P\ui{n}\in C^\infty([0, 1];\Mcal)
		\end{align*} 
		such that $P\ui{n}(0) = (\xi\ui{n}|A\ui{n})$ and $P\ui{n}(1) = (\xi\ui{n+1}| A\ui{n+1})$ for $n=1, \ldots, N-1$. 
		It is convenient to reparametrize $P\ui{n}$ in the following way: 
		
		Fix $0<\beta<\tfrac{1}{2}$ (to be specified later) and let $\psi\in C^\infty([0,1];[0,1])$ be a transition function that vanishes in a neighbourhood of $0$, takes the value $1$ close to $1$ and satisfies $|\psi'| \leq 2$. 
 For $\eps>0$ sufficiently small, we set 
		\begin{align*} 
			P_\epsilon(t) := (P\ui{n}\circ \psi)\bigl(\tfrac{t-t\ui{n}}{\eps^\beta}\bigr) \qquad \text{for $t\in[t\ui{n},t\ui{n}+\eps^\beta]$}
			\end{align*}
			for $n=1, \ldots, N-1$. Regarding the scaling behavior of $P_\eps$ and its derivatives one finds that
		\begin{align}\label{scaling_peps}
			\norm{P_\eps }_{C^0(\Gamma_\eps;\R^{3\times 3})}  = \Ocal(1),\quad  \norm{P_\eps'}_{C^0(\Gamma_\eps;\R^{3\times 3})} = \Ocal(\eps^{-\beta}) \quad \text{and} \quad \norm{P_\eps''}_{C^0(\Gamma_\eps;\R^{3\times 3})} = \Ocal(\eps^{-2\beta}),
		\end{align}
		where $\Gamma_\eps:=\bigcup_{n=1}^{N-1} [t\ui{n}, t\ui{n} + \eps^\beta]$.

		Now, let $J\subset J'\subset \R$ be closed and bounded intervals as in Lemma \ref{lem:reparam_det=1} and $\overline\omega\subset J\times J$.
		With inspiration from~\cite[Theorem~4.5]{Sca06}, we define an auxiliary sequence $(v_\eps)_{\eps}$ of functions on the cuboid 
		$Q_L':=[0,L]\times J'\times J'$; precisely, for $\eps>0$ and $x\in Q_L'$,
		\begin{align*}
			v_\epsilon(x)=\begin{cases} 
			(\xi\ui{1}|A\ui{1}) x_\eps + b_\eps\ui{1} &\text{ if } x_1\in[t\ui{0},t\ui{1}),\\[0.2cm]
				\displaystyle\int_{t\ui{n}}^{x_1} P_\epsilon (t) e_1 \dd t &\text{ if } x_1\in[t\ui{n},t\ui{n}+\eps^\beta) \text{ with $n=1, \ldots, N-1$},
				\\ \qquad \quad +\  P_\epsilon(x_1)(x_\eps -x_1e_1)+ d_\eps\ui{n} & \\[0.2cm]
				(\xi\ui{n}|A\ui{n}) x_\eps + b_\eps\ui{n} &\text{ if } x_1\in[t\ui{n}+\eps^\beta,t\ui{n+1}) \text{ with $n=1, \ldots N-1$};
			\end{cases}		
		\end{align*}	
		here, $x_\eps :=(x_1, \eps x_2, \eps x_3)$, and the translation vectors $b_\epsilon\ui{n}, d_\epsilon\ui{n}\in \R^3$ are chosen in such a way that $v_\epsilon$ is continuous. 
		It is immediate to see that 
		\begin{align}\label{convergence12}
			v_\eps\to u\qquad \text{ uniformly  in $\overline \Omega$ as $\eps\to 0$.}
		\end{align}
		
Let us collect some further useful properties of the functions $v_\eps$. 
		In fact, $v_\eps$ is not only continuous, but by construction even smooth, so in particular, $v_\eps\in C^2(Q_L';\R^3)$,
		 and a calculation of the rescaled gradients gives 
		\begin{align*}
			\rn{v} (x)= \begin{cases} 
				(\xi\ui{1}|A\ui{1})&\text{ if } x_1\in[t\ui{0},t\ui{1})\\
				P_\epsilon (x_1)&\text{ if }x_1\in[t\ui{n},t\ui{n}+\eps^\beta) \text{ with $n=1, \dots, N-1$, }\\ \quad + \ \eps  P_\epsilon'(x_1) (x_2e_2+x_3e_3)\otimes e_1 & \\ 
			 (\xi\ui{n}|A\ui{n})&\text{ if } x_1\in[t\ui{n}+\eps^\beta, t\ui{n+1}) \text{ with $n=1, \dots, N-1$. }
			\end{cases}
		\end{align*}
	 	Since $\beta<\frac{1}{2}$, the sequence $(\rn{v})_{\eps}$ is bounded in $C^0(Q_L';\R^{3\times 3})$. 
		Moreover, the function $v_\eps$ satisfies the incompressibility condition exactly except on sets of small measure, 
		where $\det \rn{v}$ is close to $1$. To quantify this statement, we compute
		\begin{align*}
			\det \rn{v}(x)= 1 + \eps  \det \bigl(P_\eps'(x_1) (x_2e_2 + x_3e_3) | P_\eps(x_1)e_2| P_\eps(x_1)e_3\bigr)
		\end{align*} for $x\in Q_\eps:=\Gamma_\eps\times J'\times J'$, and observe that 
		\begin{align}\label{130}
			\det \rn{v} =1 \quad \text{on $Q_L'\setminus Q_\eps$. }
		\end{align} 
		Thus, it follows in view of~\eqref{scaling_peps} that 
		\begin{align}\label{hypo_lemma}
			\|\det \rn{v} -1\|_{C^0(Q_L')}= \Ocal(\epsilon^{1-\beta})\quad \text{ and } \quad \|\nabla( \det \rn{v})\|_{C^0(Q_L';\R^3) }= \Ocal(\eps^{1-2\beta}). 
			\end{align}

		Now, with~\eqref{hypo_lemma} at hand,  we are in the position to apply Proposition~\ref{lem:reparam_det=1} to the sequence $(v_\eps)_{\eps}$ with $\gamma = 1-2\beta$ to obtain 
		a modified sequence $(u_\eps)_{\eps}\subset C^1(\overline\Omega;\R^3)$ that satisfies $\det \rn{u} = 1$ everywhere in $\Omega$, namely
		\begin{align*}
			u_\epsilon(x) := v_\epsilon(x_1,x_2,\ffi_\eps(x)), \quad x\in \Omega,
		\end{align*}
		with $\ffi_\eps\in C^1(Q_L;J')$ such that \eqref{est_varphi} holds.
		Notice that the inner perturbation defining $u_\eps$ corresponds to the identity map on $Q_L\setminus Q_\eps$,
 		since, due to~\eqref{130}, the ordinary differential equation in \eqref{ODE} reduces to $\partial_3 \varphi_\epsilon = 1$ on this set; thus, along with~\eqref{139},
		\begin{align}\label{136}
			u_\eps=v_\eps \quad \text{and}\quad \rn{u} =\rn{v} \in L_0(W)\subset L_0(W_0)\qquad  \text{on $\Omega\setminus Q_\eps$.}
		\end{align}
		Furthermore, as a consequence of~\eqref{est_varphi},
		\begin{align}\label{137}
			\| u_\eps-v_\eps\|_{C^0(\overline \Omega;\R^{3\times 3})} = \Ocal(\eps^{2-2\beta}) \quad \text{and}\quad \| \rn{u}-\rn{v}\|_{C^0(\overline \Omega;\R^{3\times 3})} = \Ocal(\eps^{1-2\beta}),
		\end{align}
		and therefore, also $(\rn{u})_{\eps}$ is bounded in $C^0(\overline{\Omega};\R^{3\times 3})$. Along with ~\eqref{convergence12}, it follows that
		\begin{align*}
			u_\eps\weakly u \quad \text{in $H^1(\Omega;\R^3)$.}
		\end{align*} 

		We are now in the position to conclude the proof Step~1 by showing that,
		\begin{align}\label{upperbound}
			\lim_{\eps\to 0} \Ical_\eps^\alpha(u_\eps)=0 =\Ical^\alpha(u).
		\end{align} 
		The two cases $\alpha < \frac{1}{2}$ and $\alpha\geq \frac{1}{2}$, call for a different reasoning, which we will detail next.

		\textit{Step 1a: The case $\alpha\geq \tfrac{1}{2}$.} Let $\beta<\frac{1}{2}-\frac{\alpha}{4}$.
		Then, joining (H3),~\eqref{136},~\eqref{scaling_peps} and~\eqref{137} with the observations that $|Q_\eps| =\Ocal(\eps^\beta)$, $\det \rn{u}=1$ 
		and $P_\eps\in \SO(3)$ pointwise, gives rise to the following estimate,
		\begin{align*}
			\Ical_\epsilon^\alpha (u_\epsilon)  & = \frac{1}{\epsilon^\alpha}\int_{\Omega} W_0(\rn{u})\dd x \leq  \frac{C_3}{\epsilon^\alpha}\int_{\Omega} \dist^2(\rn{u}, \SO(3)) \dd x  \\ 
				& \leq \frac{2C_3}{\eps^\alpha}\left(\int_{ Q_\eps} \dist^2(\rn{v}, \SO(3)) \dd x  + \int_{Q_\eps \cap \Omega} |\rn{v}-\rn{u}|^2\dd x\right) \\ 
				& \leq \frac{2C_3}{\epsilon^\alpha}\left( \eps^2|J|^4|Q_\eps|\norm{P_\eps'}_{C^0(\Gamma_\eps;\R^{3\times 3})}^2  + |Q_\eps| \norm{\rn{v}-\rn{u}}_{C^0(\overline \Omega;\R^{3\times 3})}^2\right)  \\
				& = \Ocal(\eps^{2-\alpha-\beta}) + \Ocal(\eps^{2-\alpha - 3\beta}) =  \Ocal(\eps^{2-\alpha - 3\beta}).
		\end{align*}
		The choice of $\beta$ yields~\eqref{upperbound}.
		
		\textit{Step 1b: The case $\alpha<\tfrac{1}{2}$.} Let $\alpha<\beta<\frac{1}{2}$.
		We invoke~\eqref{136} and (H2), as well as ~$\det \rn{u}=1$ in $\Omega$, $|Q_\eps| =\Ocal(\eps^\beta)$, and the uniform boundedness of $\rn{u}$ 
		to infer that
		\begin{align*}
			\Ical_\epsilon^\alpha (u_\epsilon)  & = \frac{1}{\epsilon^\alpha}\int_{\Omega} W_0(\rn{u})\dd x \leq \frac{1}{\epsilon^\alpha}\int_{Q_\eps\cap \Omega} W_0(\rn{u}) \dd x \\ 
				& \leq  \frac{C_2}{\epsilon^\alpha} \int_{Q_\eps\cap \Omega} |\rn{u}|^2 +1\dd x  
				\leq  C_2 |Q_\eps|\epsilon^{-\alpha}\bigl( \norm{\rn{u}}_{C^0(\overline{\Omega};\R^{3\times 3})}^2 + 1\bigr)
				=\Ocal(\eps^{\beta-\alpha}).
		\end{align*}
				
		\textit{Step 2: Relaxation and approximation.} 
		To address the general case, let $u\in H^1(0, L;\R^3)$ such that $u'\in L_0(\overline{W})^{\rm c}$ a.e.~in $(0, L)$.  	 
		By standard tools from convex and asymptotic analysis (cf. e.g.~Caratheodory's theorem and the Riemann-Lebesgue lemma), 
		there is a sequence $(u_j)_{j}\subset A_{\rm pw}(0, L;\R^3)$ such that $u_j'\in L_0(\overline W)$ a.e.~in $(0, L)$ and
		\begin{align*}
			u_j\weakly u\quad \text{in $H^1(0, L;\R^3)$.}
		\end{align*}
		Now, Step~1 applied for each fixed $j\in\N$ provides sequences
		$(u_{j,\epsilon})_{\epsilon}\subset C^1(\overline{\Omega};\R^3)$ with the properties that $u_{j,\epsilon}\weakly u_j$  in $H^1(\Omega;\R^3)$ and $\lim_{\epsilon\to 0}\Ical_\epsilon^\alpha(u_{j,\epsilon}) = 0$. Extracting a diagonal sequence $(u_\eps)_{\eps}$ with the help of a generalized version of Attouch's diagonalization lemma (see e.g.~\cite[proof of Proposition~1.11 (p.~449)]{FeF12}) finally gives the sought after recovery sequence for $u$.
	\end{proof}	

	\section*{Appendix}
	\begin{proof}[Proof of Lemma~\ref{prop:bezier_mollifying}]	
		The idea of the proof is to mollify $u$ using B\'ezier curves with sufficiently many control points to ensure the desired $C^k$-regularity. This way, the mollified curve has derivatives that lie in the convex hull of two neighbouring slopes of $u$. However, if the latter happen to be anti-parallel, then, by design, the derivative of the mollified curve vanishes at some point. To circumvent this issue, we perturb $u$ via
		a suitable loop construction. 
		
		Since $u\in A_{\rm pw}(0,L;\R^3)$ is piecewise affine and $u'\neq 0$ almost everywhere, there is a partition 
		$0=: t\ui{0} < t\ui{1} < \ldots < t\ui{N-1} < t\ui{N} := L$ 
		of the interval $[0, L]$ and vectors $\xi\ui{n}\in \R^3\setminus \{0\}$ such that 
		\begin{align}\label{uandxi}
		u' = {\xi}\ui{n} \quad  \text{on $(t\ui{n-1}, t\ui{n})$\quad for $n=1, \ldots, N$.}
		\end{align}
		
	\textit{Step~1: The case without reversions.} First, we will prove the statement under the assumption that neither two $\xi\ui{n}$ and $\xi\ui{n+1}$ from~\eqref{uandxi} are anti-parallel. Without loss of generality, it suffices to detail the case $N=2$, where $u'$ takes only the two values $\xi\ui{1}$ and $\xi\ui{2}$. For general $N$, one can simply repeat the same construction.  
		
		\textit{Step 1a: Definition of suitable B\'ezier curves.}
		For $\eta>0$ sufficiently small, we choose $2k+1$ 
		control points around $u(t\ui{1})$ by  
		\begin{align}\label{controlpoints}
		u(t_{\eta, m}) \text{ with $t_{\eta, m} := t\ui{1} - (k-m)\tfrac{\eta}{k}$ for $m=0,...,2k$. }
		\end{align}
		Then, 
		\begin{align}\label{choice_control_points}
				u(t_{\eta, m+1}) - u(t_{\eta, m}) =\begin{cases}
				\tfrac{\eta}{k} \xi\ui{1} &\text{ if } m\in\{0,\ldots,k-1\},\\ 
				\tfrac{\eta}{k} \xi\ui{2} &\text{ if } m\in \{k,\ldots, 2k-1\}.
			\end{cases}
		\end{align}
		Based on the control points in~\eqref{controlpoints}, we consider the B\'ezier curve $B_\eta: \R\to \R^3$ by
		\begin{align}\label{def:Bezier}
			B_\eta(t) = \sum_{m=0}^{2k} b_{m,2k}(t) u(t_{\eta,m}), \qquad  t\in\R,
		\end{align}
		where $b_{q,p}:\R\to\R$ are the Bernstein polynomials, cf.~Lemma \ref{lem:bernstein}. 
		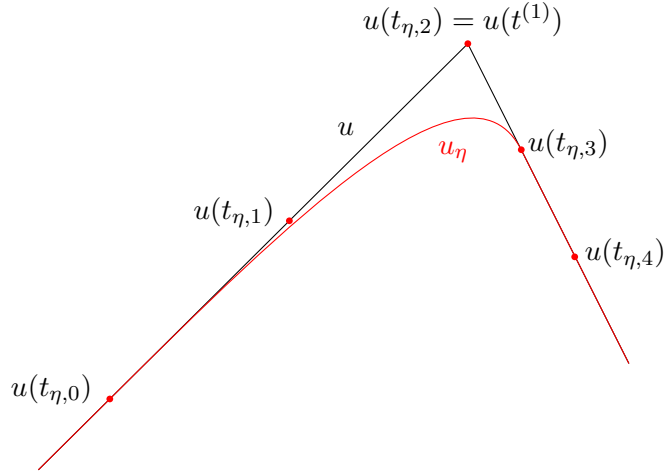
\begin{figure}[h!]
			\centering
			\begin{tikzpicture}[scale =0.5]
				\node (bezier) at (0,0) {\includegraphics[width=0.5\textwidth]{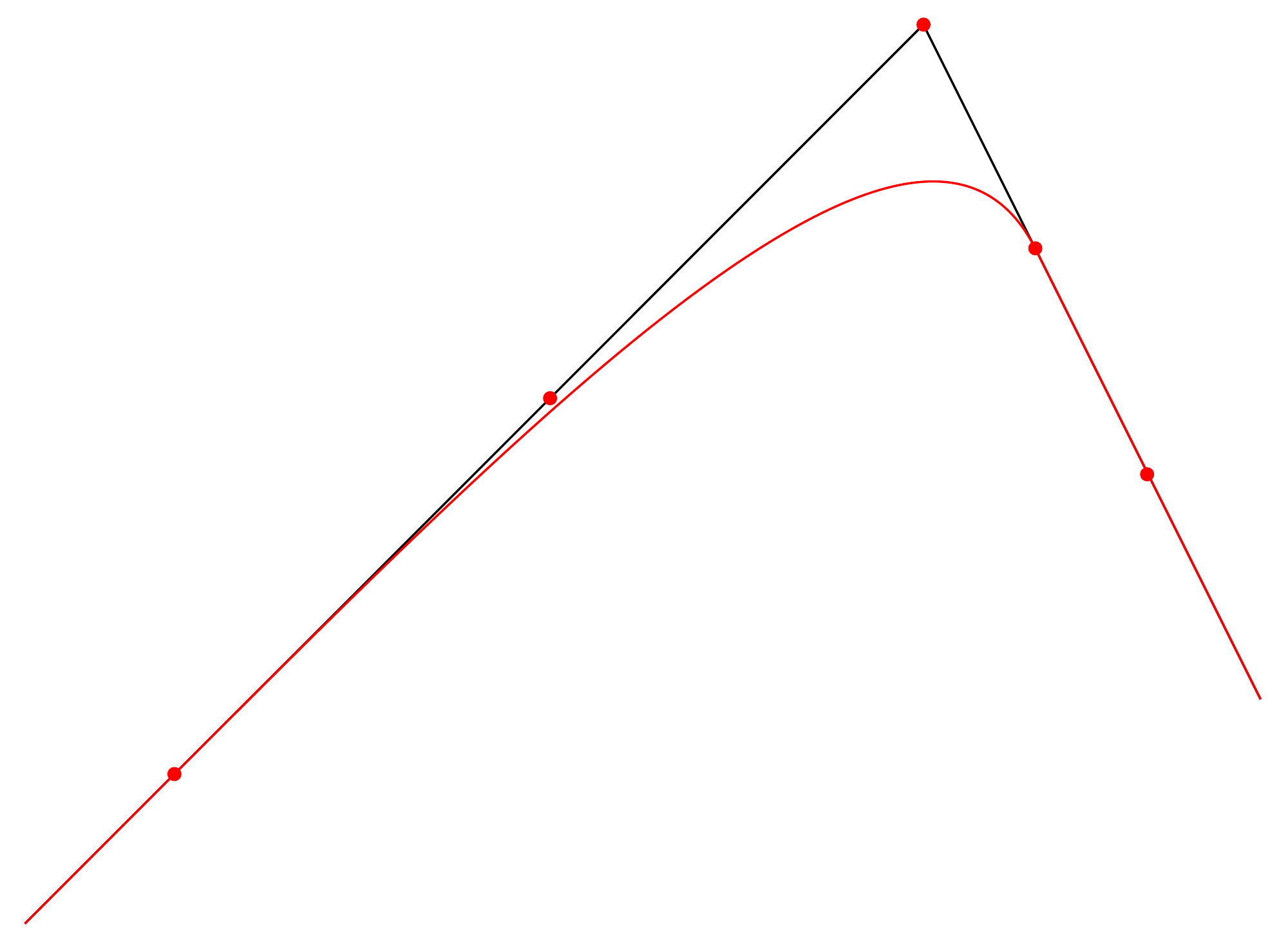}};
				\draw (-6.1,-3.5) node [anchor=east] {$u(t_{\eta,0})$};
				\draw (-1.3,1.2) node [anchor=east] {$u(t_{\eta,1})$};
				\draw (6.3,6.3) node [anchor=east] {$u(t_{\eta,2}) = u(t\ui{1})$};
				\draw (7.5,3) node [anchor=east] {$u(t_{\eta,3})$};
				\draw (9,0.1) node [anchor=east] {$u(t_{\eta,4})$};
				\draw (0.8,3.4) node [anchor=east] {$u$};
				\draw (3.8,2.8) node [red,anchor=east] {$u_\eta$};
			\end{tikzpicture}
			\caption{Illustration of $u_\eta$ 
			 for the case $k=2$.}\label{Bezier}
		\end{figure}
		
		After suitable reparametrization,~\eqref{def:Bezier} provides a mollification of $u$ via
		\begin{align}\label{ueta}
			u_{\eta}(t)= \begin{cases}
				B_{\eta}\bigl(\tfrac{t-t_{\eta, 0}}{2\eta}\bigr), &\text{ for $t\in [t\ui{1}-\eta,t\ui{1} +\eta]=[t_{\eta, 0}, t_{\eta, 2k}]$,} \\ u(t), &\text{ otherwise, } \\
			\end{cases}\qquad t\in [0,L],
		\end{align}	
			see Figure~\ref{Bezier}.
				
\textit{Step 1b: Regularity of $u_\eta$.} 
		Next, we verify that $u_{\eta}$ as constructed in Step~1a is indeed $k$-times continuosly differentiable on $[0,L]$. 
Indeed, it is enough to check that
		\begin{align}\label{derivative1}
		B_\eta'(0) = 2\eta \xi\ui{1} \quad \text{and}\quad B_\eta'(1) = 2\eta\xi\ui{2}, 
		\end{align}
		and that for any $j\in \N$ with $2\leq j \leq k$, 
		\begin{align}\label{derivativej}
		\frac{\dd^j}{\dd t^j}B_\eta(0)  = \frac{\dd^j}{\dd t^j}B_\eta(1)= 0.
		\end{align}
		As for~\eqref{derivative1}, we obtain with the help of Lemma~\ref{lem:bernstein}\,a),\,c) and \eqref{choice_control_points}  
		that for all $t\in \R$,
		\begin{align}\label{Bprime}
			B_\eta'(t) &= 2k \sum_{m=0}^{2k} (b_{m-1,2k-1}(t) - b_{m,2k-1}(t)) u(t_{\eta,m}) \nonumber\\
						&= 2k\sum_{m=0}^{2k-1}b_{m,2k-1}(t)\left(u(t_{\eta,m+1}) - u(t_{\eta,m})\right)\\
					&= 2\eta \Bigl(\sum_{m=0}^{k-1}b_{m,2k-1}(t)\xi\ui{1} + \sum_{m=k}^{2k-1}b_{m,2k-1}(t)\xi\ui{2}\Bigr) \nonumber\\
&= 2\eta (\lambda(t)\xi\ui{1} + (1-\lambda(t))\xi\ui{2}), \nonumber 
		\end{align}
where $\lambda (t):= \sum_{m=0}^{k-1} b_{m,2k-1}(t) \in [0,1]$ for $t\in \R$.
	Due to~Lemma~\ref{lem:bernstein}\,b), $\lambda(0) = 1$ and $\lambda(1) = 0$, which yields~\eqref{derivative1}. 
	
	Similar calculations, invoking again the properties of Bernstein polynomials, in particular Lemma~\ref{lem:bernstein}\,d), 
	give~\eqref{derivativej}. 
		
		\textit{Step 1c: Uniform bounds and convergence of $(u_\eta)_{\eta}$.} 
		As a consequence of~\eqref{Bprime}, the first derivative of 
			 $u_\eta$ stays within the line segment connecting $\xi\ui{1}$ and $\xi\ui{2}$, or in other words, is a convex combination of these two vectors; formally, 
	\begin{align*} 
	u_\eta'\in [\xi\ui{1}, \xi\ui{2}]:= \{\xi\in \R^3: \xi= \lambda\xi\ui{1} + (1-\lambda) \xi\ui{2} \text{with $\lambda\in [0,1]$}\}. 
	\end{align*}

		Since $\xi\ui{1}$ and $\xi\ui{2}$ are not antiparallel, it follows that $0\notin [\xi\ui{1}, \xi\ui{2}]$. Hence, in view of the compactness of the line segment $[\xi\ui{1}, \xi\ui{2}]$, 
 one can find constants $c,C>0$ independent of $\eta$ such that 
		\begin{align*}
			c\leq |u_{\eta}'(t)| \leq C
		\end{align*}
		for all $t\in [0, L]$.
		Moreover, along with \eqref{ueta}, 
		\begin{align*}
			\int_{0}^L |u' - u_{\eta}'|^2 \dd{t} \leq 2\eta (C + |\xi\ui{1}| + |\xi\ui{2}|)^2,  
		\end{align*} 
		and therefore,
		\begin{align*}
			u_{\eta} \to u \text{ in $H^{1}(0,L;\R^3)$ as $\eta\to 0$}
		\end{align*}
		by Poincar\'e's inequality.
	After passing to a suitable discrete sequence, this completes the proof of statement under the assumption that the curve $u$ is free of reversion. 
	
 \textit{Step 2: The general case with reversions.} The idea is to reduce the argument to the situation of Step~1 via a loop construction and to conclude with a diagonalization argument.  

In the following, let $I$ stand for the index set consisting of all $n\in \{1, \ldots, N-1\}$ such that $\xi\ui{n}$ and $\xi\ui{n+1}$ are anti-parallel, that is,
		\begin{align*}
			\xi\ui{n+1} = -\nu_n \xi\ui{n} 
		\end{align*} 
		for some $\nu_n> 0$. 
		
		\textit{Step~2a: Loop construction.} Without loss of generality, $I$ is a singleton, say $I=\{1\}$; otherwise the argument below can be performed analogously for all (finitely many) elements in $I$. Besides, as in Step~1, we take $N=2$ to keep notations simple. 
	
	For $\delta>0$ sufficiently small, define $u_\delta\in A_{\rm pw}(0, L;\R^3)$ via linear interpolation such that 
		\begin{align}\label{def_udelta}
			u_\delta = u\quad \text{on $[0, t\ui{1}-\delta] \cup [t\ui{1} +\delta \sigma, L]$} \quad \text{and}\quad 
						u_\delta(t\ui{1})	= u(t\ui{1}) + \delta \xi\ui{1}_\perp,
		\end{align}
			where $\xi\ui{1}_\perp$ is a non-zero vector orthogonal to $\xi\ui{1}$, and $\sigma = 1$ if $\nu_1\neq 1$ and $\sigma = \frac{1}{2}$ if $\nu_1=1$, 
		see Figure~\ref{schlaufe}. 
		
		\begin{figure}[h!]
			\centering
			\begin{tikzpicture}[scale=0.8]
				\draw (0,0) -- (5,5);
				\fill[black] (5,5) circle (0.5mm) node[anchor=south] {$u(t\ui{1})$};
				\fill[black] (3,3) circle (0.5mm) node[anchor=south east] {$u(t\ui{1}+\delta)$};
				\fill[black] (1,1) circle (0.5mm) node[anchor=south east] {$u(t\ui{1}-\delta)$};
				\draw[ultra thick,->,red] (1,1)--(1.8,1.8) node[anchor = north west] {$\xi\ui{1}$};
				\draw[ultra thick,->,red] (3,3) node[anchor =north  west] {$\xi\ui{2}$} --(2.5,2.5);
				\draw (2,5) node {a)};
				
				\begin{scope}[shift={(8,0)}]
					\draw (0,0) -- (3,3);
					\draw[dashed] (3,3) -- (5,5);
					\draw[dashed] (5,5) -- (7,3);
					\draw (3,3) -- (7,3);
					\draw (1,1) -- (7,3);
					\draw (2,5) node {b)};
					\draw[ultra thick,->,red] (0,0)--(0.8,0.8) node[anchor = east] {$\xi\ui{1}$};
					\draw[ultra thick,->,red] (3,3)--(2.5,2.5) node[anchor = east] {$\xi\ui{2}$};
					\draw[ultra thick,->,red] (1,1)--++(18.43:0.707cm) node [anchor = north west] {$\xi\ui{1}+\xi\ui{1}_\perp$};
					\draw[ultra thick,->,red] (7,3)--(5.7,3);
					\draw (4.8,3) node [red,anchor=south] {$\xi\ui{2} -\xi\ui{1}_\perp$};
					\fill[black] (5,5) circle (0.5mm);
					\fill[black] (7,3) circle (0.5mm) node[anchor=north] {$u_\delta(t\ui{1})$};
					\fill[black] (3,3) circle (0.5mm) node[anchor=south east] {$u_\delta(t\ui{1}+\delta)$};
					\fill[black] (1,1) circle (0.5mm) node[anchor=south east] {$u_\delta(t\ui{1}-\delta)$};
				\end{scope}
			\end{tikzpicture}
			\caption{ Illustration of  a) a curve $u$ that reverses its path at the point $t\ui{1}$; b) the modified curve $u_\delta$ resulting from the loop construction in the case $\nu_1 \neq 1$.}\label{schlaufe}
		\end{figure}
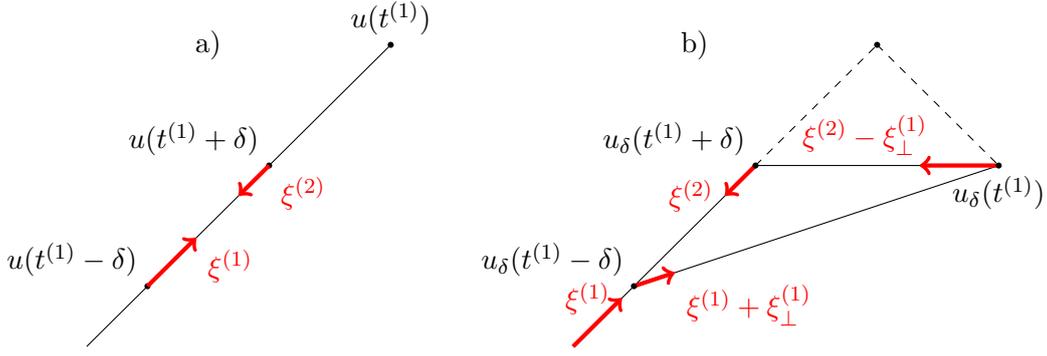
  By design, 
		\begin{align*} 
		u_\delta'\in \{\xi\ui{1}, \xi\ui{2}, \xi\ui{1} + \xi_\perp\ui{1},  \xi\ui{2} - \tfrac{1}{\sigma} \xi\ui{1}_\perp \}\quad \text{a.e.~in $[0, L]$,}
		\end{align*}
so that in particular, 
\begin{align}\label{bound_udelta}
c\leq \|u_\delta'\|_{C^0([0,L];\R^3)}\leq C,
\end{align} 
with constants $c, C>0$ depending only on $u$, and thus, independent of $\delta$. 
		Therefore, since $u_\delta$ differs from $u$ only on a set of measure $\delta(1+\sigma)$, we conclude together with Poincar\'e's inequality that
		\begin{align}\label{convergence_udelta}
		u_\delta \to u\quad \text{ in $H^1(0,L;\R^3)$ as $\delta\to 0$. } 
		\end{align}
		
 \textit{Step 2b: Diagonalization.} 
		By applying the results of Step~1 to each $u_\delta$ and accounting for~\eqref{bound_udelta} and~\eqref{def_udelta}, we obtain sequences $(u_{\delta, i})_i\subset C^k([0, L];\R^3)$ such that 
				\begin{itemize}
		\item[$(i)_\delta$] $c\leq \|u_{\delta, i}'\|_{C^0([0,L];\R^3)}\leq C$ for all $i\in \N$ and $\delta>0$ with constants $c,C>0$ depending on $u$;
		\item[$(ii)_\delta$] $u_{\delta, i} = u \text{ on } [0,L]\setminus \Gamma_{2i}^\delta$, where
		 $\Gamma_i^\delta: = \{t\in [0,L] : \dist(t,\Gamma^\delta) \leq \tfrac{1}{i}\}$ for $i$ sufficiently large and $\Gamma^\delta$ denotes the set of points in $(0, L)$ where $u_\delta$ is not differentiable;
		\item[$(iii)_\delta$] $u_{\delta, i} \to u_\delta$ in $H^{1}(0,L;\R^3)$ as $i\to \infty$;
		\end{itemize}	
		In consideration of~\eqref{convergence_udelta}, $(ii)_\delta$ and $(ii)_\delta$, we can pick a diagonal sequence $(u_i)_i\subset C^k([0,L];\R^3)$ with $u_i:= u_{\delta(i), i}$ such that $\Gamma_{2i}^{\delta(i)} \subset \Gamma_i$ for all $i\in \N$, and $u_i \to u$ in $H^{1}(0,L;\R^3)$ as $i\to \infty$ by Attouch's lemma, which proves the statement.  
	\end{proof}
	 
The next lemma gathers some basic facts about Bernstein polynomials, which were an important ingredient in the definition of B\'ezier curves in the previous proof. For more details, we refer the reader e.g.~to \cite{BDS11, Far02}. 

	\begin{lemma}\label{lem:bernstein} For $q\in \Z$ and $p\in \N$, 
	 let $b_{q,p} : \R\to \R$  be the corresponding Bernstein polynomial, i.e.,
		\begin{align*}
			b_{q,p}(t) =\begin{cases} \displaystyle {{p}\choose{q}} (1-t)^{p-q}t^q & \text{if $q\leq p$ and $q\geq 0$,}\\ 0 & \text{otherwise,}
			\end{cases} \qquad  t\in\R.
		\end{align*}

	 Then the following properties hold: 
		\begin{itemize}
			\item[a)] (binomial theorem)  
				\begin{align*} 
					\sum_{m=0}^{p} b_{m,p} = 1;
				\end{align*}
					\item[b)] (values at $0$ and $1$) 
				\begin{align*} 
					b_{q,p}(0) = \begin{cases} 1 & \text{if $q=0,$}\\ 0 & \text{if $q\neq 0$,}\end{cases}\quad \text{and} \quad b_{q,p}(1) = \begin{cases} 1 & \text{if $q=p$,}\\ 0 & \text{if $q\neq p$;}\end{cases} 
				\end{align*}  
			\item[c)] (first derivative) 
				\begin{align*}
					b_{q,p}' = p(b_{q-1,p-1}-b_{q,p-1});
				\end{align*}
			\item[d)] (higher-order derivatives) 
				\begin{align*}
					\frac{\dd^j}{\dd t^j}b_{q,p}= \frac{p!}{(p-j)!}\sum_{m=\max\{0,q-p+j\}}^{\min\{j,q\}} (-1)^{m+j} {{j}\choose m} b_{q-m,p-j} 
				\end{align*} 
				for any natural number $j\leq p$.  
		\end{itemize}
	\end{lemma}

\subsection*{Acknowledgements} 
CK acknowledges partial financial support by the UU Westerdijk fellowship program. 

\bibliographystyle{abbrv}
\bibliography{Dimensionreduction}
\end{document}